\documentclass[11pt,reqno]{amsart}
\usepackage{amscd,amssymb,amsmath,amsthm}
\usepackage{mathtools}
\usepackage[english]{babel}
\usepackage[T1]{fontenc}
\usepackage[arrow,matrix]{xy}
\usepackage{graphicx,subfigure,tikz,tikz-3dplot}
\usetikzlibrary{positioning,matrix,arrows,calc}
\usepackage{cite}
\usepackage{dsfont}
\usepackage{enumitem}
\usepackage{bbm}
\usepackage{hyperref}
\hypersetup{
	colorlinks=true,
	linkcolor=blue}
\usepackage{array}

\oddsidemargin=0.3in \evensidemargin=0.3in \theoremstyle{plain}
\newtheorem{theorem}{Theorem}[section]
\newtheorem{lemma}[theorem]{Lemma}
\newtheorem{proposition}[theorem]{Proposition}
\newtheorem{corollary}[theorem]{Corollary}

\theoremstyle{definition}
\newtheorem{remark}[theorem]{Remark}
\newtheorem{definition}[theorem]{Definition}
\newtheorem{example}[theorem]{Example}
\newtheorem{question}[theorem]{Open Question}

\numberwithin{equation}{section}
\newcommand{\N}{\mathbb{N}}
\newcommand{\Z}{\mathbb{Z}}
\newcommand{\R}{\mathbb{R}}
\newcommand{\C}{\mathbb{C}}
\newcommand{\D}{\mathbb{D}}
\newcommand{\DD}{\mathcal{D}}
\newcommand{\mF}{\mathcal{F}}
\newcommand{\mO}{\mathcal{O}}
\newcommand{\dg}{\mathbf{m}}
\newcommand{\proj}{\mathrm{proj}}

\newcommand{\Sc}{\mathbb{S}^1}
\newcommand{\Sph}{\mathbb{S}^3}
\newcommand{\TT}{\mathcal{T}}
\newcommand{\Ts}{\mathcal{T}^+}
\newcommand{\hF}{\widehat{F}}

\newcommand{\cwnew}{w}
\newcommand{\cznew}{z}

\newcommand{\intrr}{\mathop{\mathrm{int}}}

\newcommand{\E}{\mathbb{E}}
\newcommand{\Prob}{\mathbb{P}}

\newcommand{\dT}{G}

\newcommand{\id}{\mathrm{id}}

\newcommand{\Ind}{\mathbbm{1}}

\newcommand{\SL}{\mathrm{SL}}
\newcommand{\Homeo}{\mathrm{Homeo}}

\newcommand{\eps}{\varepsilon}
\newcommand{\Eta}{H}
\newcommand{\wEta}{\widetilde{\Eta}}
\newcommand{\weta}{\widetilde{\eta}}
\newcommand{\wS}{\widetilde{S}}

\newcommand{\tF}{\widetilde{F}}

\newcommand{\mgr}{\mu}
\newcommand{\msp}{\nu}
\newcommand{\hmsp}{\widehat{\msp}}
\newcommand{\wto}{\xrightarrow{\text{weakly}}}

\newcounter{hypocounter}
\renewcommand\thehypocounter{(H\arabic{hypocounter})}
\title{Atomic physical measures for non-invertible random dynamical systems}
\author{Vincent P. H. Goverse}
\thanks{V.G. was supported by the Scientific High Level Visiting Fellowships (SSHN) of the Higher Education, Research and Innovation Department of the French Embassy in the United Kingdom and the EPSRC
Centre for Doctoral Training in Mathematics of Random Systems: Analysis, Modelling and Simulation (EP/S023925/1). }
\address{Department of Mathematics, Imperial College London}
\email{
vincent.goverse@duke.edu, \quad
vincent.goverse21@imperial.ac.uk}
\author{Victor Kleptsyn}
\thanks{V.K. was supported in part by ANR Gromeov (ANR-19-CE40-0007) and by Centre Henri Lebesgue (ANR-11-LABX-0020-01).}
\address{CNRS, Institute of Mathematical Research of Rennes, UMR 6625 du CNRS, University of Rennes}
\email{victor.kleptsyn@univ-rennes.fr}



\begin{document}
\begin{abstract}
    We construct an example of a random dynamical system on the circle, formed by maps that are only locally invertible, which possesses an atomic stationary measure~$\nu$. Moreover, this measure is physical: for Lebesgue-almost every initial point $x_0$, the Cesàro averages of its random trajectory almost surely converge to~$\nu$. This shows that the H\"older regularity of stationary measures, known for (non-measure-preserving) random dynamical systems formed by diffeomorphisms, cannot be generalized to this class of systems.
     We also provide some related examples, including ones where a stationary measure charges a proper submanifold, despite the absence of a closed common invariant submanifold.
\end{abstract}

\maketitle

\tableofcontents

\section{Introduction}

\subsection{Initial question and statement of the main result}

A classical argument shows that the stationary measures of random dynamical systems formed by homeomorphisms that do not admit common invariant finite sets are non-atomic. This argument is straightforward: assuming the opposite, it suffices to note that the set of atoms of highest weight must be a finite invariant set, giving a contradiction.
However, it heavily relies on the injectivity of the dynamics (though it was recently noticed in~\cite{mu1} that this condition can be weakened to the so-called $\mu$-injectivity, introduced in \cite{mu0}).

This non-atomicity statement can be generalized to the regularity control of stationary measures, established in different situations by many authors; see the works of Y. Guivarc'h, R. Aoun, Y.~Benoist, J.-F.~Quint~\cite{guivarch_produits_1990,benoist_random_2016,aoun_random_2020}. In particular, a recent result by A. Gorodetski, G. Monakov and the second author \cite{gorodetski_holder_2022} states that stationary measures for random dynamical systems formed by diffeomorphisms satisfy the H\"older regularity condition, under the assumptions that there are no common invariant measures and some moment conditions are satisfied. The latter result was also generalized by Monakov~\cite{monakov_log-holder_2024} for bi-H\"older homeomorphisms (instead of diffeomorphisms) or weaker moment conditions, establishing a weaker (log-H\"older) regularity estimate for the measure.

Let us recall that H\"older regularity of a measure refers to the following: there exist constants $C,\alpha>0$ such that
\[
    \mu(B_\varepsilon(x)) \leq C \varepsilon^\alpha
\]
for all $x$ and all $\varepsilon>0$. The supremum of such exponents $\alpha$ is sometimes called the Frostman dimension of $\mu$ \cite{MR4520027}.

Given all these generalizations, it seemed tempting to extend the results in yet another direction, namely to the non-invertible case, assuming that the maps are only \emph{locally} invertible. The purpose of this note is to show that this is actually impossible, providing an example of an \emph{atomic} stationary measure occurring in this setting. Namely, we have the following theorem.

\begin{theorem}[Main result]\label{thm:circle}
There exists a random dynamical system $(f_1,\dots,f_s;p_1,\dots,p_s)$ where $f_i$ are $C^{\infty}$ locally invertible maps of the circle $\Sc=\R/\Z$, and $p_i>0$ are the probabilities of their application, and a stationary measure
\[
\msp=\sum p_i (f_i)_* \msp,
\]
such that:
\begin{itemize}
\item The maps $f_i$ admit no common invariant measure, that is, there is no probability measure $\msp'$ on the circle such that for all $i=1,\dots,s$ one has $(f_i)_*\msp'=\msp'$.
\item The measure $\msp$ consists of a countable number of atoms,
\begin{equation}\label{eq:nu-atoms}
\msp=\sum_{j=0}^{\infty} m_j \delta_{a_j}, \quad a_j\in \Sc.
\end{equation}
\item Moreover, the measure $\msp$ is a ``physical'' measure of this system, in the sense that the distribution of a random trajectory of a Lebesgue-generic point converges to it. That is, for Lebesgue-a.e.\ $x_0\in \Sc$ and for a.e.\ sequence $\omega\in \{1,\dots,s\}^{\N}$ (w.r.t. the Bernoulli measure), for the trajectory
\[
x_{j+1}=f_{\omega_j}(x_{j}), \quad j=0,1,2,3,\dots,
\]
one has
\[
\frac{1}{n} \sum_{j=0}^{n-1} \delta_{x_j} \wto  \msp, \quad n\to \infty.
\]
\end{itemize}
\end{theorem}

\begin{remark}
There is an example of a random dynamical system $(f_1,f_2;p_1,p_2)$, formed by two maps and satisfying the conclusions of Theorem~\ref{thm:circle}, where $f_1$ is a map of degree~$2$, and $f_2$ is a circle diffeomorphism. Graphs of two such maps $f_1,f_2$ are shown in Fig.~\ref{fig:GS-example} below.
\end{remark}

While Theorem~\ref{thm:circle} shows that Hölder regularity fails in general for
locally invertible maps, it is natural to ask whether it can be recovered under
additional assumptions.

\begin{question}
    Under what conditions on a random dynamical system formed by locally invertible
    (but not necessarily globally invertible) maps can one conclude that every
    stationary measure satisfies a Hölder regularity estimate?
\end{question}

\begin{figure}[ht]
    \centering
\includegraphics[height=4.5cm]{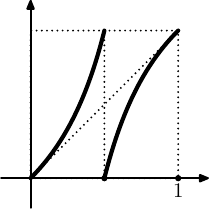} \qquad \includegraphics[height=4.5cm]{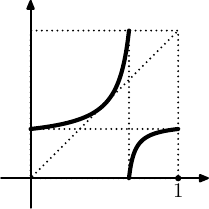}
    \caption{Smooth example for Theorem~\ref{thm:circle}: maps~$f_1$ (left) and $f_2$ (right).}
    \label{fig:GS-example}
\end{figure}

\subsection{Further discussion: measures charging submanifolds}

Another direction for generalization of Hölder estimates of stationary measures is to obtain an upper bound for a stationary measure of an $\eps$-neighbourhood of a proper \emph{submanifold} instead of a point. Such estimates are known due to Benoist-Quint~\cite{benoist_random_2016} for the stationary \emph{projective} dynamics for neighbourhoods of hyperplanes. Nonstationary analogues of such estimates would allow one to obtain a Central Limit Theorem for nonstationary products of matrices in arbitrary dimension (actually, this statement is only known for products of matrices in $\SL(2,\R)$, see~\cite{gorodetski-etal-2024-CLT}; in the latter case, the projective dynamics is one-dimensional, so the hyperplanes are actually points).

Our main example leads to other examples showing that, in full generality, a stationary measure need not satisfy Hölder-type bounds in a neighbourhood of a (closed) submanifold.
In fact, a submanifold can be of positive measure, even if the dynamics is invertible.
Such examples can be constructed using the same solenoid technique that is used to transform the circle doubling map into an invertible map (see for example \cite[Chapter 17.1]{katok_introduction_1995}). Namely, we have the following example.
\begin{example}\label{ex:torus}
Take $M$ to be the solid torus, $M=\Sc\times \D$, where $\D=\{|\cwnew|\le 1\}$ is the unit disc in~$\C$, and let $F_1,\dots,F_s: M\to M$ be defined by
\begin{equation}\label{eq:F-j-def}
    F_j(x,\cwnew) = (f_j(x), \eps \cwnew + \frac{1}{2} e^{2\pi i x}), \quad j=1,\dots,s,
\end{equation}
where $f_j$ are the maps from Theorem~\ref{thm:circle}. Then:
\begin{itemize}
\item For all sufficiently small $\eps$ the maps $F_j$ are diffeomorphisms onto their images.
\item There is no common $F_j$-invariant measure, and there is no common $f_j$-invariant measure on the space of discs $\{\D_a\}_{a\in\Sc}$.
\item There is a stationary measure $\hmsp$ on $M$ that is supported on a countable union $\bigcup_j \D_{a_j}$ of these discs, where $a_j$ are the atoms of the measure $\msp$ from the conclusion~\eqref{eq:nu-atoms} of Theorem~\ref{thm:circle}.
\item Moreover, the measure $\hmsp$ is the physical measure of this system: for Lebesgue-almost every point $b_0=(x_0,\cwnew_0)\in M$, the distribution of its random trajectory $b_{n+1}=F_{\omega_n}(b_{n})$ almost surely converges to~$\hmsp$,
\begin{equation}\label{eq:convergence-hmsp}
\frac{1}{n} \sum_{j=0}^{n-1} \delta_{b_j} \wto\hmsp, \quad n\to\infty.
\end{equation}
\end{itemize}
\end{example}
Example~\ref{ex:torus} falls within the class of those considered in~\cite{gorodetski_holder_2022} (as it now consists of invertible maps), and thus its stationary measure satisfies the H\"older regularity property; however, the measure of the discs~$\D_a$ may be positive. In particular, the measure of a neighbourhood of a disc~$\D_{a_j}$ does not tend to zero. However, this example has one drawback: the solid torus itself is not a closed manifold, and the maps are only diffeomorphisms onto their images.

We postpone the discussion of technical details for this and the following examples (that is, Examples~\ref{ex:torus}--\ref{ex:sphere2}) to Section~\ref{s:examples-proofs} below, for now concentrating on the general picture.

To address the non-closedness drawback, consider the solid torus $M=\Sc\times \D$ as a subset of the 3-sphere $\mathbb{S}^3$. As we will see in Section~\ref{s:examples-proofs} below, there is a modification of the maps in Example~\ref{ex:torus}, that is formed by skew products of the form
\begin{equation*}
\hF_j(x,\cwnew)=(f_j(x),L_x(\cwnew)+\frac{1}{2}e^{2\pi i x}),
\end{equation*}
where $L_x$ are sufficiently strongly contracting linear maps, for which all the maps $\hF_j:M\to M$ can be extended to diffeomorphisms $\tilde{F}_j:\mathbb{S}^3\to \mathbb{S}^3$. Moreover, one can ensure that for this extension the random trajectories, starting from any point, almost surely fall into~$M$. Hence, we have the following modification of Example~\ref{ex:torus}:

\begin{example}\label{ex:sphere}
There exist diffeomorphisms $\tF_j:\Sph\to \Sph$, probabilities $p_j$ of their application and a stationary measure $\hmsp$ of the random dynamical system $(\tF_1,\dots,\tF_s;p_1,\dots,p_s)$, such that:
\begin{itemize}
\item The measure $\hmsp$ is supported inside the solid torus $M=\Sc\times \D \subset \Sph$, and therein is supported on a countable number of discs,
\[
\hmsp \left( \bigcup\nolimits_j \D_{a_j}\right) =1.
\]
\item The maps $\tF_i$ admit no common invariant measure,
and there is no common $\tF_i$-invariant measure on the space of discs $\{\D_a\}_{a\in\Sc}$.
\item The measure $\hmsp$ is a ``physical'' measure of this system: for Lebesgue-a.e.\ initial point $b_0\in \Sph$ and almost every random
trajectory $b_{n+1}=\tF_{\omega_n}(b_{n})$,
one has
\[
\frac{1}{n} \sum_{j=0}^{n-1} \delta_{b_j} \wto  \hmsp, \quad n\to \infty.
\]
\end{itemize}
\end{example}

We conclude with two examples of a different nature. Namely, in Examples~\ref{ex:torus} and~\ref{ex:sphere} the stationary measure of a two-dimensional proper submanifold was positive (despite the absence of a common invariant manifold). One can ask whether there exists an example for which the stationary measure of some one-dimensional submanifold is positive, again despite the absence of a common invariant manifold. The following two examples provide a positive answer to this question.

Take the middle-third standard Cantor set $C\subset [0,1]$, invariant under the maps
\[
\hat f_1: x \mapsto \frac{1}{3}x, \quad \hat f_2: x \mapsto \frac{x+2}{3}.
\]
Taking $p_1=p_2=1/2$, one gets a unique $(\hat f_1,\hat f_2;p_1,p_2)$-stationary
measure~$\msp$ on~$C$.
Now, there is a $C^{\infty}$-function $\theta:[-1,2]\to [0,1/2]$, vanishing exactly on the set $C$,
\[
\theta^{-1}(\{0\})=C
\]
and with derivative $|\theta'(x)| < 1/2$
(such a function exists for any closed subset of~$\R$); see Fig.~\ref{fig:theta}.
\begin{figure}
    \centering
    \includegraphics[width=0.8\linewidth]{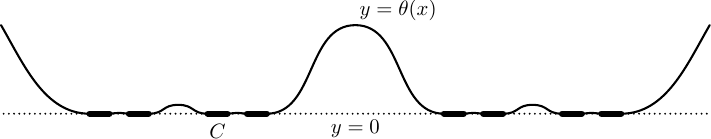}
    \caption{Cantor set $C$, supporting the stationary measure $\msp$, and the (vertically scaled) graph of the function $y=\theta(x)$.}
    \label{fig:theta}
\end{figure}
The following example is an analogue of Example~\ref{ex:torus}:
\begin{example}\label{ex:rectangle}
Take four maps of the rectangle $\Pi=[-1,2]\times [-1,1]$ to itself,
\[
F_{i,j}(x,y) = (\hat f_i(x), \tfrac{1}{2} y + (-1)^j \theta(x)), \quad i,j=1,2,
\]
and take the probabilities of their application to be
$p_{i,j}=\tfrac{1}{4}$. These maps are contracting diffeomorphisms onto their
respective images, and the unique stationary measure $\hmsp$  of the random dynamical system~$(F_{i,j};p_{i,j})$ is the push-forward of the measure~$\msp$ under the inclusion map $x\mapsto (x,0)$ from $[0,1]$ to $\Pi$.
This system has
no common invariant measure, and no proper submanifold is invariant under all~$F_{i,j}$. Meanwhile, the measure~$\hmsp$ is supported on the line $\{y=0\}$.
\end{example}

In the same way as in Example~\ref{ex:sphere}, one can consider $\Pi$ as a subset of the sphere $\mathbb{S}^2$ and extend the maps $F_{i,j}$ from Example~\ref{ex:rectangle} to diffeomorphisms $\tF_{i,j}$ of the sphere. Again, we can ensure sufficient North--South behaviour so that the (random) trajectory of every initial point almost surely falls into $\Pi$ after some number of iterations. This leads to the following example.

\begin{example}\label{ex:sphere2}
There exist diffeomorphisms $\tF_{i,j}:\mathbb{S}^2\to \mathbb{S}^2$ and a stationary measure $\hmsp$ of the random dynamical system $(\tF_{i,j};p_{i,j})_{i,j=1,2}$, such that:
\begin{itemize}
\item The measure $\hmsp$ is supported inside the rectangle $\Pi=[-1,2]\times [-1,1]$, and therein is supported on the line $\{y=0\}$;
\item The maps $\tF_{i,j}$ admit no common invariant measure. Moreover, there is no common $\tF_{i,j}$-invariant closed submanifold of~$\mathbb{S}^2$.
\item The measure $\hmsp$ is the unique stationary measure of this system. In particular, for every
initial point $b_0\in \mathbb{S}^2$ and almost every random
trajectory $b_{n+1}=\tF_{\omega_n}(b_{n})$, one has
\[
\frac{1}{n} \sum_{j=0}^{n-1} \delta_{b_j} \wto \hmsp, \quad n\to \infty.
\]
\end{itemize}
\end{example}

Given, on one hand, the theorems on the regularity of stationary measures that are already known, and on the other hand, Examples~\ref{ex:torus}--\ref{ex:sphere2} above, it is very interesting to address the following question.

\begin{question}
    What are the best (or most natural) conditions under which one can conclude that a stationary measure $\msp$ of a random dynamical system does not charge any proper closed submanifold~$N$? Or under which, for any closed submanifold~$N$, the measure of its $\eps$-neighbourhood~$U_{\eps}(N)$ satisfies a H\"older-type bound $\msp(U_{\eps}(N))< C \eps^{\alpha}$ for some positive constants~$C,\alpha$?
\end{question}

\subsection{Plan of the paper and sketch of the proof}
We recall a few general notions in Section~\ref{sec:prelim}. Then, we start the proof of Theorem~\ref{thm:circle} in Section~\ref{sec:thompson} by first constructing a piecewise affine example: a random dynamical system that already possesses an atomic stationary measure, though there is no common invariant measure for its maps.

Such an example is obtained by combining the doubling map $\varphi:x\mapsto 2x \mod 1$ with some other maps from the Thompson group $T$, with the choice of the probabilities that strongly favours the application of~$\varphi$. As a result, we will note that it possesses an atomic stationary measure, supported on the set of dyadic rational points of the circle (that is, the orbit of the point $x_0=0$). Moreover, we consider random dynamical systems formed by Thompson-like maps, and establish in Theorem~\ref{thm:atomic} a sufficient condition for such a system to possess an atomic stationary measure. For that, it suffices that the expectation of the logarithm of the derivative on each of the affinity intervals is strictly positive.

Next, in Section~\ref{ss:G-S}
we apply the Ghys--Sergiescu smooth realization technique (see~\cite{ghys_sur_1987}). Namely, the maps from the Thompson semigroup are piecewise formed by compositions of the circle doubling map $\varphi$ and its inverse branches. Replacing $\varphi$ by a degree-two map $\psi$ that is $C^{\infty}$-tangent to the identity at its fixed point $0$, one transforms the semigroup into a semigroup of $C^{\infty}$ locally invertible maps. This allows us to consider smooth random dynamical systems, still possessing atomic stationary measures (supported on the image of the dyadic rational points under the conjugating map).

Then, the conjugated example piecewise consists of compositions of $\psi$ and branches of its inverse, and the random dynamics favours applying $\psi$ over its inverse. It is known~\cite{Thaler1980,MR1452184} that under the iterations of the map $\psi$ (non-strictly expanding and having a parabolic fixed point~$0$) the distribution of a Lebesgue-generic point $x_0$ converges to the Dirac measure $\delta_0$; roughly speaking, it is relatively easy for a point to be sent near~$0$, while it takes a lot of time to get away from it. Now, the \emph{random} iterations that favour $\psi$ over its inverses mostly follow the mere iterations of $\psi$, occasionally going back along one of the preimages; hence, the distribution of such a random trajectory is almost surely concentrated on the full orbit of~$0$. We formalize this hand-waving argument in Section~\ref{ss:physical}, thus concluding that the constructed measure~$\msp$ is indeed physical, and completing the proof of Theorem~\ref{thm:circle}.

We present some concluding remarks in Section~\ref{ss:concluding}: for a particularly symmetric version of our example, it turns out that for any invariant measure of $\varphi$ we can describe the corresponding stationary measure of this symmetric version explicitly; see \eqref{eq:nu-Q}.

Finally, Section~\ref{s:examples-proofs} is devoted to the technical details for Examples~\ref{ex:torus}--\ref{ex:sphere2}.

\section{Preliminaries}\label{sec:prelim}

\subsection{Random dynamical systems}

In this paper, we interpret a \emph{random dynamical system} as a measure $\mgr$ on the space of (continuous) maps from some set~$X$ to itself. (Naturally, there are other ways to introduce it, as a cocycle or a skew product over a shift, as well as many generalizations, which we leave outside the scope of the present work). One then considers the sequence of random compositions
\begin{equation}\label{eq:Markov}
F_n=g_n\circ\dots\circ g_1 = g_n \circ F_{n-1}, \quad F_0=\id,
\end{equation}
where the maps $g_i$ are chosen i.i.d. with respect to the measure~$\mgr$; in other words, the sequence $(g_1,\dots,g_n,\dots)$ is distributed w.r.t. the measure $\mathbb{P}=\mgr^{\N}$.

In the context of this note, the measure $\mgr$ will always be finitely supported: we will choose a finite number of maps $f_1,\dots,f_s$ and the probabilities $p_i>0$ of their application, satisfying $p_1+\dots+p_s=1$. Then on each step we will independently choose to apply one of the maps $f_i$ with the corresponding probabilities; we will abbreviate this by saying that we are considering the random dynamical system $(f_1,\dots,f_s;p_1,\dots,p_s)$. In the general context above, it corresponds to choosing the measure
\[
\mgr= p_1 \delta_{f_1}+ \dots +  p_s \delta_{f_s};
\]
the compositions~\eqref{eq:Markov} can then be written as
\[
F_n=f_{\omega_{n-1}}\circ \dots \circ f_{\omega_0},
\]
where $\omega_j$ are i.i.d. random variables, taking value $i=1,\dots,s$ with the probability~$p_i$.

A measure~$\msp$ on~$X$ is \emph{stationary} for the random dynamical system given by the measure~$\mgr$ if it is equal to the average of its pushforward images,
\[
\msp = \int (f_* \msp) \, d\mgr(f).
\]
In our case, this corresponds to the relation
\[
\msp = \sum_{j=1}^s p_j (f_j)_* \msp.
\]
This condition is equivalent to saying that the measure $\msp$ is a stationary measure for the Markov chain $x_{n}=g_n (x_{n-1})$ with the maps $g_i$ chosen independently w.r.t. the measure~$\mgr$.

\subsection{Thompson group}\label{s:prelim-Thompson}

Denote by $\DD$ the set of dyadic rational points on the circle: let
\[
\DD=\Z\left[\tfrac{1}{2}\right]/\Z \subset \Sc = \R/\Z.
\]

The following group was introduced by R.~J.~Thompson in unpublished handwritten notes; see~\cite{MR1426438, ghys_sur_1987}.

\begin{definition}
    The \emph{Thompson group}~$\TT$ is a group of orientation-preserving homeomorphisms of the circle $\Sc=\R/\Z$, consisting of maps~$f$ that satisfy the following conditions:
\begin{itemize}
    \item $f$ is piecewise-affine, and on each interval of affinity the map $f$ has the form
    \begin{equation}\label{eq:dyadic}
        f(x)=2^m x + c
    \end{equation}
    for some $m\in \Z$, $c=\frac{a}{2^q}\in \DD$.
    \item Endpoints of the affinity intervals are dyadic rationals.
\end{itemize}
\end{definition}
Next, following a seminal paper by \'E.~Ghys and V.~Sergiescu~\cite{ghys_sur_1987}, let us recall how this group can also be described in terms of the doubling map
\begin{equation}
\label{eq:phi-0}
    \varphi(x) = 2x \mod 1.
\end{equation}

To do so, one first notices that the point~$0$ is the fixed point of $\varphi$. The set of dyadic rational points of the circle is the full (that is, including taking preimages) $\varphi$-orbit of~$0$. These points, joined by edges corresponding to the application of $\varphi$, form a tree with a loop attached at its root~$0$: see Fig.~\ref{fig:tree}, left.
\begin{figure}[ht!]
    \centering
    \includegraphics[width=0.37\linewidth]{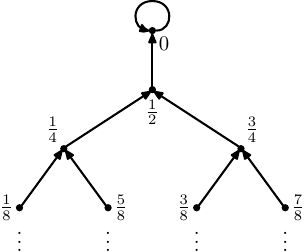}
    \quad
    \includegraphics[width=0.5\linewidth]{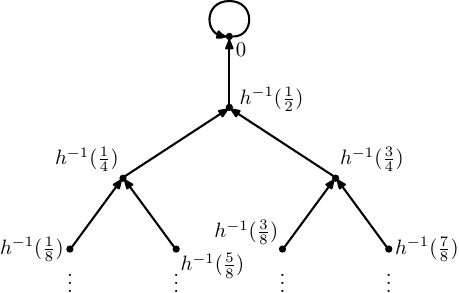}
    \caption{Left: graph of a dyadic tree representing the preimages of zero under the doubling map \( \varphi \). Right: preimages of $0=\psi(0)$ under the map \( \psi \); arrows correspond to the application of~$\psi$. }

    \label{fig:tree}
\end{figure}
Finally, condition~\eqref{eq:dyadic} is equivalent to saying that, on each of its affinity intervals, the map $f$ can be written as a composition of iterations of $\varphi$ and appropriately chosen branches of its inverse.

\section{Piecewise linear maps}\label{sec:thompson}
\subsection{A piecewise linear example}

We start by constructing a partial example towards Theorem~\ref{thm:circle}. Namely, we will construct a system of \emph{piecewise-affine} locally invertible maps on the circle, for which there is no common invariant measure and which admits an atomic stationary measure. We have the following example:

\begin{example}\label{ex:PL}
Consider a random dynamical system $(f_1,f_2;p_1,p_2)$, where
\begin{align*}
    f_1(x) &= 2x \mod 1 ,\\
    f_2(x) &=
    \begin{cases}
    \frac{x}{2}+\frac{1}{4}, &\text{ for } x \in [0,1/2), \\
    2x-\frac{1}{2}, &\text{ for } x \in [1/2,3/4), \\
    x - \frac{3}{4}, &\text{ for } x \in [3/4,1);
    \end{cases}
\end{align*}
see Fig.~\ref{fig:PL-example}. Assume that the probabilities of their application are such that $p_1>\frac{1}{2}$; for instance, one can take
\[
p_1=\frac{3}{4},\quad p_2=1-p_1 = \frac{1}{4}.
\]
Then, as we will see below (see Theorem~\ref{thm:atomic}), this random dynamical system admits an atomic stationary measure, though it has no common invariant measure.
\end{example}

\begin{figure}[ht!]
    \centering
\includegraphics[height=4.5cm]{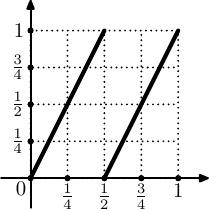} \qquad
\includegraphics[height=4.5cm]{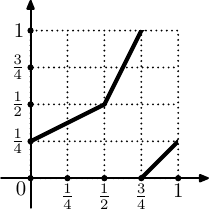}
    \caption{Piecewise affine example for Example~\ref{ex:PL}: maps~$f_1$ (left) and $f_2$ (right).}
    \label{fig:PL-example}
\end{figure}

The absence of a common invariant measure for the system described in Example~\ref{ex:PL} can be checked in a straightforward way:
\begin{lemma}\label{lem:nocommon}
    There is no common invariant measure for $f_1$ and $f_2$.
\end{lemma}
\begin{proof}
The map $f_2$ has a unique fixed point $x_0=\frac{1}{2}$, and the orbit of any other point converges to it under the iterations of~$f_2$. Hence, the only possible candidate for a common invariant measure is the Dirac measure~$\delta_{1/2}$. However, the point~$\frac{1}{2}$ is not fixed by~$f_1$, and hence this measure is not $f_1$-invariant. This implies that there are no common $(f_1,f_2)$-invariant measures.
\end{proof}

Now, the existence of an atomic stationary measure for the system described in Example~\ref{ex:PL} actually fits into a more general context, which we will describe in the next section.

\subsection{General framework: Thompson semigroup}
Let us introduce the following class of maps:
\begin{definition}
    The \emph{Thompson semigroup}~$\Ts$ is a semigroup of orientation-preserving, continuous and piecewise-affine maps of the circle $\Sc=\R/\Z$, consisting of maps $f$ that satisfy the following conditions:
    \begin{itemize}
        \item On each interval of affinity the map $f$ has the form
        \begin{equation}\label{eq:dyadic-aff}
            f(x)=2^m x + c
        \end{equation}
        for some $m\in \Z$, $c\in \DD$.
        \item Endpoints of the affinity intervals are dyadic rationals.
    \end{itemize}
\end{definition}

It is easy to see that these maps indeed form a semigroup, and that the action of any of these maps preserves the set of dyadic rational points of the circle~$\DD$, in the same way as for the Thompson group. The \emph{Thompson group}~$\TT$ is the subgroup of the invertible elements of this semigroup:
\[
\TT=\{ f\in \Ts \mid f \in \Homeo_+(\Sc)\}.
\]

Now, consider a random dynamical system $(f_1,\dots,f_s;p_1,\dots,p_s)$, formed by such maps. We introduce the following property.

\begin{definition}
    A random dynamical system $(f_1,\dots,f_s;p_1,\dots,p_s)$, formed by maps $f_i$ from the Thompson semigroup, is \emph{expanding on average} if for every point $x\in\Sc\setminus \DD$ one has
    \[
        \sum_j p_j \log f_j'(x) >0.
    \]
    Choosing a partition of the circle into common affinity intervals $J_i, \, i=1,\dots,N$ for all the maps $f_j$ simultaneously, and writing for each of these intervals
    \begin{equation}\label{eq:simultaneous}
        f_j(x) = 2^{m_{ij}}x +c_{ij}, \quad x\in J_i, \quad j=1,\dots,s,
    \end{equation}
    this condition can be equivalently rewritten as
    \begin{equation}\label{eq:drift}
        \forall i=1,\dots,N \quad \sum_{j=1}^s p_j m_{ij}>0.
    \end{equation}
\end{definition}

For such systems, we have the following result:

\begin{theorem}\label{thm:atomic}
    Assume that a random dynamical system $(f_1,\dots,f_s;p_1,\dots,p_s)$ is formed by maps from the Thompson semigroup $f_1,\dots,f_s\in \Ts$, and that this system is expanding on average. Then this system admits an atomic stationary measure~$\msp$, supported on the set of dyadic rational points~$\DD$. If, moreover, the doubling map $\varphi$ is one of the maps~$f_i$, then the atomic stationary measure supported on $\DD$ is unique.
\end{theorem}

\begin{remark}
    The system, described in Example~\ref{ex:PL}, satisfies the assumptions of Theorem~\ref{thm:atomic}, and hence admits an atomic stationary measure, supported on~$\DD$.
\end{remark}

\subsection{Random dynamics on \texorpdfstring{$\DD$}{D}}

This section is devoted to the proof of Theorem~\ref{thm:atomic}.

Let $\dT:\DD\to \R_{\geq 0}$  be the function on the set of dyadic rationals, defined by
\[
\dT(x) = \min\{n\ge 0: \varphi^n(x)=0\} = \begin{cases}
    0, & x=0 \\
    n, & x=\frac{2k-1}{2^n}.
\end{cases}
\]
In other words, $\dT(x)$ is the distance from the vertex $x$ to the root~$0$ in the tree shown in Figure~\ref{fig:tree}.

\begin{proof}[Proof of Theorem~\ref{thm:atomic}]
    Note first that by the definition of the Thompson semigroup $\Ts$, each of the maps $f_i$ preserves the set of the dyadic rationals~$\DD$. Hence, we can consider the restriction to the initial random dynamical system on this (countable) invariant set, that we now consider with a discrete topology.

    The following lemma states that the induced random walk on~$\DD$ has a drift towards the root.

    \begin{lemma}\label{cl:decreasingn}
        Assume that a system $(f_1,\dots,f_s;p_1,\dots,p_s)$ satisfies the assumptions of Theorem~\ref{thm:atomic}. Then there exist $r>0$ and $C_0\geq1$ such that for every $x\in \DD$ with $\dT(x)>C_0$ one has
        \begin{equation}\label{eq:submartingale}
            \E_{\mgr} \dT(f(x)) \le \dT(x) -r
        \end{equation}
    \end{lemma}

    \begin{proof}
        Note that for a map $f(x)=2^m x + \frac{a}{2^q}$ and a point $x=\frac{b}{2^n}\in \DD$, one has
        \[
            f(x)= \frac{b}{2^{n-m}}+ \frac{a}{2^q},
        \]
        and hence $\dT(f(x))=\dT(x)-m$, once $n=\dT(x)>q+m$.

        Consider now intervals~$J_i$, $i=1,\dots, N$, on each of which all the maps $f_j$ are simultaneously affine and are given by~\eqref{eq:simultaneous}. For each such interval, there exists $C_i$ such that once $x\in J_i$ and $\dT(x)>C_i$, one has
        \[
            \forall j=1,\dots,s \quad \dT(f_j(x))=\dT(x) - m_{ij}.
        \]
        Now, $r_i:=\sum_j p_j m_{ij}$ is positive due to the assumption of expansion on average~\eqref{eq:drift}. Finally, take $r:=\min_i r_i>0$ and $C_0:=\max_i C_i$. Then, for every $x\in \DD$ with $\dT(x)>C_0$, taking an interval $J_i \ni x$, one gets the desired
        \[
            \E_{\mgr} \dT(f(x)) = \sum_{j=1}^s p_j \cdot (\dT(x) -m_{ij}) = \dT(x) - r_i \le \dT(x)-r. \qedhere
        \]
    \end{proof}

    This drift towards the root suffices to conclude that the Markov chain $x_{n+1}=f_{\omega_n}(x_{n})$ on the (countable) set $\DD$ admits a stationary measure. We start with an informal explanation: as we will see, in that case, the (random) hitting time of a finite ``core'' part $\{x\in \DD \mid \dT(x)\le C_0\}$ has exponentially decreasing tails. In particular, this tail bound guarantees that the distribution of random trajectories, starting from $x_0\in \DD$, stays tight. Thus, a stationary measure can be found by the standard averaging procedure: considering Cesàro averages of the distributions of $x_t$, and taking any accumulation point of these measures (the tightness guarantees that it will be a probability measure, as no mass can ``escape to infinity'').

    The formal way to ensure such a drift is checking what is called a \emph{Foster--Lyapunov drift condition} -- consisting of existence of a \emph{Margulis function} (the terminology slightly varies in different domains).
    Roughly speaking, a \emph{Margulis function} is a function that sufficiently ``decreases on average'' outside some compact part of the space of states. Its existence implies that the chain returns to this compact part sufficiently often, thus guaranteeing the existence of a stationary measure. See, for example, Theorem~1.3.1 in Meyn--Tweedie~\cite{MR1287609}. In our setting, this condition states that if there exists a function $V:\DD\to {[1,\infty)}$ such that
    \begin{enumerate}
        \item\label{i:compact} $V^{-1}([1,c])$ is compact for every $c\in [1,\infty)$,
        \item\label{i:gamma} there exist $\gamma<1$ and $C$ such that $\E_{\mgr} V(f(x)) \le \gamma \cdot V(x)$ once $V(x)>C$,
    \end{enumerate}
    then there exists a stationary probability measure on~$\DD$. If, in addition, the chain is irreducible, then it is positive recurrent and the stationary probability measure is unique.

    The following lemma checks this condition:
    \begin{lemma}\label{l:applying-FL}
        For all sufficiently small $\theta>0$, the function~$V:\DD\to {[1,\infty)}$, defined by
        \begin{equation}\label{eq:V-def}
            V(x)=e^{\theta \dT(x)},
        \end{equation}
        satisfies the conditions~\eqref{i:compact} and~\eqref{i:gamma} above.
    \end{lemma}
    \begin{proof}
        Condition~\eqref{i:compact} is immediate. To establish condition~\eqref{i:gamma}, in the same way as in Lemma~\ref{cl:decreasingn}, note that for all $x\in J_i$ with $\dT(x)>C_i$ one has
        \[
            \E_{\mgr} V(f(x)) = \E_{\mgr} e^{\theta \dT(f(x))} =
            \left( \sum_j p_j e^{-\theta m_{ij}}\right) \cdot  e^{\theta\dT(x)}.
        \]
        Let $\gamma_i(\theta):=\sum_{j=1}^s p_j e^{-\theta m_{ij}}$. Then as $\theta\to 0$ the Taylor expansion of the exponential gives
        \[
            \gamma_i(\theta) = 1 - \theta \sum_j p_j m_{ij} + o(\theta) = 1 - r_i \theta + o(\theta).
        \]
        In particular, for all sufficiently small $\theta$ one has $\gamma_i(\theta)<1$, and thus the same applies to
        $\gamma(\theta):=\max_i \gamma_i(\theta)$.
        Fixing a sufficiently small $\theta>0$ such that the inequality $\gamma=\gamma(\theta)<1$ holds, we obtain the desired
        \[
            \E_{\mgr} V(f(x)) \le \gamma V(x)
        \]
        for all $x$ with $V(x)>C$, where $C=e^{\theta C_0}$.
    \end{proof}

    The Foster--Lyapunov condition (see, for example,~\cite[Theorem~1.3.1]{MR1287609}) provides the existence of a stationary measure for the dynamics on~$\DD$.
    Now, if $\varphi\in \{f_1,\dots, f_s\}$, then $0$ can be reached with positive probability by the trajectory starting from any point $x\in \DD$. Hence, the Markov chain on the orbit of~$0$ under the semigroup generated by $\{f_1,\dots, f_s\}$ is irreducible. This irreducibility, together with the positive recurrence provided by the Foster--Lyapunov condition, implies the uniqueness of the stationary measure on~$\DD$. This concludes the proof of Theorem~\ref{thm:atomic}.
\end{proof}

We conclude this section with the exponential tail estimate that we will use in the proof of Theorem~\ref{thm:circle}.

Take an arbitrary starting point $x_0\in \DD$ and consider the Markov chain $x_{n+1}=f_{\omega_n}(x_{n})$; let
\[
T:= \min\{ n\ge 0 \mid \dT(x_n) \le C_0\}
\]
be the (random) first hitting
time of the set
\[
X_0:=\{x\in \DD \mid \dT(x)\le C_0\}.
\]

\begin{lemma}\label{l:exptail}
    Under the assumptions of Theorem~\ref{thm:atomic}, there exists $\beta>0$ such that
    for every $x_0\in \DD$ there exists $C=C_{x_0}$ such that for the Markov process $(x_n)$, starting from $x_0$, one has
    \[
        \Prob(T>n) \le C e^{-\beta n}.
    \]
    In particular, the first hitting time $T$ is finite almost surely.
\end{lemma}
\begin{proof}
    Fix $\theta>0$ and corresponding $\gamma<1$, provided by Lemma~\ref{l:applying-FL}.
    Consider the family of random variables
    \[
        \xi_n:= \gamma^{-n} V(x_n) \cdot \Ind_{T(\omega)>n}.
    \]
    Then these variables form a nonnegative supermartingale (with respect to the natural family of $\sigma$-algebras $\mF_n$, generated by $\omega_0,\dots,\omega_{n-1}$): either $T(\omega)\le n$, and then $\xi_n(\omega)=\xi_{n+1}(\omega)=0$, or
    \[
        \E(\xi_{n+1}\mid \mF_n) \le \gamma^{-(n+1)} \E (V(x_{n+1})\mid \mF_n) \le \gamma^{-(n+1)} \cdot \gamma V(x_n) \le \xi_n.
    \]
    In particular, the sequence of expectations $\E \xi_n$ is non-increasing. As $V(x)\ge 1$ for all $x\in\DD$, one has $\xi_n(\omega) \ge \gamma^{-n}$ whenever $T(\omega)>n$, and the desired exponential tail estimate follows from Markov's inequality:
    \[
        \Prob(T> n) \le \frac{\E \xi_n}{\gamma^{-n}} \le \gamma^n\cdot \E \xi_0 = V(x_0)\cdot \gamma^n.
    \]
    This concludes the proof.
\end{proof}

\section{Smooth maps: concluding the construction}\label{sec:conjugation}
\subsection{Smoothing the example: Ghys--Sergiescu realization}\label{ss:G-S}
The next step is to use the smooth realization of the Thompson group, constructed by \'E.~Ghys and V. Sergiescu \cite{ghys_sur_1987} (also see the exposition in~\cite[Chapter~1]{MR2809110}). Let us recall this construction.

As was already mentioned in Section~\ref{s:prelim-Thompson}, the elements of the Thompson group can be locally written as a composition of the doubling map $\varphi$ and its inverse branches; indeed, one has
\[
2^m x+ \frac{a}{2^q} = 2^{-q}\cdot (2^{m+q} x + a) = \varphi^{-q} (\varphi^{m+q}(x)) \mod 1
\]
for an appropriate choice of the branch of~$\varphi^{-q}$.

Now, let $\psi:\Sc\to\Sc$ be such that
\begin{itemize}
    \item the map $\psi$ is $C^r$-tangent to the identity at its fixed point~$0$,
    \item the map $\psi$ is a $C^r$ locally invertible map of degree~$2$,
\end{itemize}
where $r\in \N\cup \{\infty\}$. One can consider the construction of the Thompson group with $\varphi$ replaced by $\psi$, denoted by~$\TT_\psi$. Namely, the maps are now piecewise compositions of powers of $\psi$ and of appropriately chosen branches of its inverse, and the endpoints of intervals and their images belong
to the set
\[
\DD_{\psi}:=\bigcup_n \psi^{-n}(\{0\}).
\]

A theorem by Ghys and Sergiescu \cite{ghys_sur_1987} states that the Thompson group $\TT_\psi$ is formed by $C^r$-diffeomorphisms, thus providing a $C^r$-smooth action of the Thompson group $\TT$ on the circle. Indeed, take any point $b\in \DD_{\psi}$, and assume that some $f\in \TT_{\psi}$ is given to the left and to the right of it by two different compositions $\psi^{-l}\circ \psi^k$ and $\psi^{-l'}\circ \psi^{k'}$, respectively (with some choices of branches for the preimages). To each of these compositions corresponds a path in $\DD_{\psi}$, where each point is joined to its $\psi$-image (see Figure~\ref{fig:tree}, right). Both these paths start at $b$ and arrive at $f(b)$.

Now, this graph is a tree with one loop attached at the root~$0$. Hence, any two paths from $b$ to $f(b)$ differ by going around this loop (forward or backward) a given number of times, which corresponds to applying $\psi$ at the point~$0$. As $\psi$ is $C^r$-tangent to the identity at~$0$, the compositions corresponding to two such paths are $C^r$-tangent to each other. The compositions defining the map $f\in \TT_{\psi}$ on all the intervals are $C^r$-locally invertible (as the map $\psi$ is) and are tangent to each other at the endpoints of these intervals, which implies that the resulting map $f$ is a $C^r$-diffeomorphism.

We note that actually the same argument applies to the Thompson semigroup. Namely, given a map $\psi$ as before, for every $f\in \Ts$ one can associate to it the map formed by the corresponding branches of compositions $\psi^{-l}\circ \psi^k$ on the intervals whose endpoints are the corresponding points of~$\DD_{\psi}$. This produces a semigroup $\Ts_{\psi}$ of $C^r$-smooth locally invertible maps (and a $C^r$-smooth action of the Thompson semigroup on the circle).
\begin{proposition}\label{p:Ts-smooth}
    Let $\psi:\Sc\to\Sc$ be a degree~$2$ map of class $C^r$ such that
    \begin{itemize}
        \item $\psi$ is $C^r$-tangent to the identity at its fixed point~$0$,
        \item the map $\psi$ is a $C^r$ locally invertible map of degree~$2$,
    \end{itemize}
    for $r\in \N\cup \{\infty\}$. Then
    the semigroup $\Ts_{\psi}$ consists of locally invertible maps of class~$C^r$.
\end{proposition}
\begin{proof}
    In their proof in~\cite{ghys_sur_1987}, Ghys and Sergiescu do not use that $f$ is a homeomorphism; it suffices to repeat the same arguments, which imply that adjacent branches are $C^r$-tangent at the common endpoints of the intervals.
\end{proof}

Now, if additionally the expansion away from 0 condition holds, $\psi'(x)>1$ at every point $x\neq 0$, then~$\psi$ is topologically conjugate to the doubling map $\varphi$: there exists $h\in \Homeo_+(\Sc)$, such that
\begin{equation}\label{eq:psi-phi-h}
 \psi = h^{-1} \circ  \varphi \circ h.
\end{equation}
This conjugacy also conjugates all the compositions of $\varphi$ and branches of its inverses to the corresponding compositions of $\psi$ and its inverses. Hence, the full action of the Thompson group becomes conjugate:
one has
\[
\TT_{\psi}=\{h^{-1}\circ f\circ h \mid f\in \TT\}.
\]
(For a general~$\psi$, the actions are semi-conjugate instead of being conjugate, and this was used by Ghys and Sergiescu to establish the rationality of the rotation numbers of all the maps from~$\TT$.)

Again, the same argument applies to the Thompson semigroup $\Ts$: assuming $\psi'(x)>1$ at every point $x\neq 0$, one has
\begin{equation}\label{eq:Ts-conjugate}
\Ts_{\psi}=\{h^{-1}\circ f\circ h \mid f\in \Ts\}.
\end{equation}

Proposition~\ref{p:Ts-smooth} together with the description~\eqref{eq:Ts-conjugate} imply the following statement.

\begin{corollary}\label{cor:smooth}
    Given a random dynamical system $(f_{1,(0)},\dots,f_{s,(0)};p_1,\dots,p_s)$, formed by maps $f_{i,(0)}\in\TT^+$, that satisfies the assumptions of Theorem~\ref{thm:atomic}, and let $\psi:\Sc\to\Sc$ be a degree~$2$ map of class $C^r$ such that
    \begin{itemize}
        \item $\psi$ is $C^r$-tangent to the identity at its fixed point~$0$,
        \item $\psi'(x)>1$ at every point $x\neq 0$,
    \end{itemize}
    for $r\in \N\cup \{\infty\}$.
    Then the map $h$, given by \eqref{eq:psi-phi-h}, conjugates this system to a random dynamical system
    \begin{equation}\label{eq:conjugates}
        (f_{1},\dots,f_{s};p_1,\dots,p_s),    \quad f_i=h^{-1}\circ f_{i,(0)}\circ h \in\TT^+_{\psi}, \quad i=1,\dots,s
    \end{equation}
    formed by locally invertible maps~$f_i$ of class $C^r$ and possessing an atomic stationary measure~$(h^{-1})_* \msp$, supported on~$\DD_{\psi}$.
\end{corollary}

\begin{remark}
    The maps shown in Figure~\ref{fig:GS-example} are obtained by such a conjugacy from the piecewise-affine maps given in Example~\ref{ex:PL} (see Figure~\ref{fig:PL-example}).
\end{remark}

After conjugation of Example~\ref{ex:PL} we have the following example:

\begin{example}\label{ex:conju}
Let $\psi$ be a $C^{\infty}$ degree~$2$ map of the circle with a fixed point $0$, and assume
    \begin{itemize}
        \item $\psi$ is $C^\infty$-tangent to the identity at its fixed point~$0$,
        \item $\psi'(x)>1$ at every point $x\neq 0$.
    \end{itemize}
    Let $h$ be the conjugacy between $\psi$ and $\varphi$, given by \eqref{eq:psi-phi-h}, and let $c=h^{-1}(\frac{1}{2})$ be a $\psi$-preimage of~$0$. Denote $\psi_0:=\psi|_{[0,c]}$, $\psi_1:=\psi|_{[c,1]}$. Then, this conjugacy, applied to the maps from Example~\ref{ex:PL}, provides the following maps:
\begin{align*}
    f_1(x) &= \psi(x),\\
    f_2(x) &=
    \begin{cases}
    \psi_0^{-1}\circ\psi_1^{-1}\circ \psi, &\text{ for } x \in [0,h^{-1}(\frac{1}{2})), \\
    \psi_1^{-1}\circ \psi \circ \psi, &\text{ for } x \in [h^{-1}(\frac{1}{2}),h^{-1}(\frac{3}{4})), \\
    \psi_0^{-1}\circ\psi_0^{-1}\circ \psi\circ \psi, &\text{ for } x \in [h^{-1}(\frac{3}{4}),1).
    \end{cases}
\end{align*}
The graphs of these maps (for a particular choice of the map $\psi$) are shown in Fig.~\ref{fig:GS-example}.

Now, consider a random dynamical system $( f_1,  f_2;p_1,p_2)$, where $p_1>\frac{1}{2}$; for instance, one can take
\[
p_1=\frac{3}{4},\quad p_2=1-p_1 = \frac{1}{4}.
\]
Then, as we will see in Section~\ref{ss:physical}, this random dynamical system admits a physical atomic stationary measure, though it has no common invariant measure. Showing this will complete the proof of Theorem~\ref{thm:circle}.
\end{example}

\subsection{Physical measure}\label{ss:physical}

Recall that a measure $\msp$ is called \emph{physical} for a smooth random dynamical system $(f_1,\dots,f_s;p_1,\dots,p_s)$ on a compact manifold $M$, if for Lebesgue-a.e.\ initial point $x_0\in M$ and a.e.\ sequence of its random iterations
\[
x_{j+1}=f_{\omega_j}(x_{j}), \quad j=0,1,2,3,\dots,
\]
(where $\omega_j$ are i.i.d., taking each value $i$ with the probability~$p_i$), one has
\[
\frac{1}{n} \sum_{j=0}^{n-1} \delta_{x_j} \wto  \msp, \quad n\to \infty.
\]

In this section, we will show that for the smooth examples, constructed in Corollary~\ref{cor:smooth}, the atomic stationary measure, supported on $\DD_{\psi}$, is a physical one. Actually, the same holds for any random dynamical system obtained from the one satisfying the assumptions of Theorem~\ref{thm:atomic} by the Ghys--Sergiescu smooth realization procedure. Namely, given an initial point $x_0\in\Sc$ and its random trajectory $x_{n+1}=f_{\omega_n}(x_{n})$, consider the sequence of time-averaged (random) measures
\begin{equation}\label{eq:nu-bar}
    \bar{\msp}_n := \frac{1}{n} \sum_{j=0}^{n-1} \delta_{x_j}.
\end{equation}
Then, the following statement holds.

\begin{theorem}\label{thm:physical}
    Let $(f_{1,(0)},\dots,f_{s,(0)};p_1,\dots,p_s)$ be a system that satisfies the assumptions of Theorem~\ref{thm:atomic}. Let $\psi:\Sc\to\Sc$ be a degree~$2$ map of class $C^r$ such that
    \begin{itemize}
        \item $\psi$ is $C^r$-tangent to the identity at its fixed point~$0$,
        \item $\psi'(x)>1$ at every point $x\neq 0$,
    \end{itemize}
    for $r\in \N\cup \{\infty\}$, $r\geq 2$.
    Let $h$ be the conjugacy between $\psi$ and~$\varphi$, given by~\eqref{eq:psi-phi-h}. Then, for the conjugated smooth system~\eqref{eq:conjugates}, for Lebesgue-a.e.\ initial point~$x_0$, almost surely all accumulation points of the sequence $\bar{\msp}_n$, given by~\eqref{eq:nu-bar}, are measures supported on~$\DD_{\psi}$.

    Moreover, if the doubling map $\varphi$ is among the $f_{i,(0)}$ (and thus $\psi$ is among the $f_i$), then the sequence $\bar{\msp}_n$ almost surely converges to the atomic stationary measure~$\msp$, supported on~$\DD_{\psi}$.
\end{theorem}

\begin{proof}

The proof of Theorem~\ref{thm:physical} proceeds in two steps. We start by considering a \emph{deterministic} dynamical system on the circle, that is given by the iterations of the map~$\psi:\Sc\to\Sc$ only. This system is non-strictly expanding with the only parabolic fixed point at the point~$0$; such systems have been studied for many decades. A classical result in this setting (established under different assumptions by many authors) is that for a Lebesgue-generic initial point, the distribution of its trajectory in the first $n$ iterations converges to the Dirac mass at the parabolic point. In other words, it means that $\delta_0$ is the physical measure of this system; see
Thaler~\cite{Thaler1980} (Theorem~1 and Corollary~2 therein),
Inoue~\cite[Theorem~1]{MR1452184}, and Holland~\cite{MR2122916}. For our argument we use the precise formulation below, adapted from Thaler's 1980 result.

\begin{theorem}[Thaler \cite{Thaler1980}]\label{t:parabolic}
    Let $\psi:\Sc\to\Sc$ be a degree~$2$ map of class $C^r$ such that
    \begin{itemize}
        \item $\psi$ is $C^r$-tangent to the identity at its fixed point~$0$,
        \item $\psi'(x)>1$ at every point $x\neq 0$,
    \end{itemize}
    for $r\in \N\cup \{\infty\}$, $r\geq 2$.
    Then for Lebesgue-a.e.\ $x_0\in \Sc$ one has
    \begin{equation}\label{eq:conv-0}
        \frac{1}{n} \sum_{j=0}^{n-1} \delta_{\psi^j(x_0)} \wto \delta_0.
    \end{equation}
\end{theorem}
\noindent For the reader's convenience, we recall a sketch of its proof at the end of this section.

The second step of the proof is Proposition~\ref{prop:acc-pts} below: we show that whenever the initial point~$x_0$ satisfies the conclusion~\eqref{eq:conv-0} of Theorem~\ref{t:parabolic},
all the accumulation points of the sequence of distributions $\bar\msp_n$ of its random trajectory are almost surely supported
on~$\DD_\psi$.

Combining these two steps proves the first claim of
Theorem~\ref{thm:physical}; the second claim (a.s.\ convergence
to~$\msp$ when $\varphi$ is among the $f_{i,(0)}$) then follows from
uniqueness of the stationary measure on~$\DD_\psi$, since almost surely any weak-*
accumulation point of $\bar\msp_n$ is a stationary measure~\cite[Theorem 4.20]{MR4559704}.

\begin{proposition}\label{prop:acc-pts}
    Let $\psi:\Sc\to\Sc$ be a degree~$2$ map of class $C^r$ such that
    \begin{itemize}
        \item $\psi$ is $C^r$-tangent to the identity at its fixed point~$0$,
        \item $\psi'(x)>1$ at every point $x\neq 0$,
    \end{itemize}
    for $r\in \N\cup \{\infty\}$, $r\geq 2$. Let $x_0\in \Sc$ be such that~\eqref{eq:conv-0} holds. Then almost surely all the accumulation points of the sequence $\bar{\msp}_n$, given by~\eqref{eq:nu-bar}, are measures supported on~$\DD_{\psi}$.
\end{proposition}
\begin{proof}
    Let $\dT_{\psi}:\DD_{\psi}\to \Z_+$ be the distance to the root function in the tree on Fig.~\ref{fig:tree}, right:
    \[
        \dT_{\psi}(x):=\dT(h(x))=\min \{n\ge 0 \mid \psi^n(x)=0\},
    \]
    and denote by $\DD_{\psi;N}$ the set of points of $\DD_{\psi}$ for which this distance does not exceed a given~$N$:
    \[
    \DD_{\psi;N}:=\{x\in \DD_{\psi}, \, \dT_{\psi}(x)\le N\}.
    \]
    It suffices to show that
    \begin{equation}\label{eq:eps-1-2}
    \forall \eps_1>0 \quad  \exists N: \quad \forall \eps_2>0 \quad \liminf_{n\to\infty} \bar{\msp}_n\left(
    U_{\eps_2}(\DD_{\psi;N})
    \right) \ge 1-\eps_1 \quad \text{a.s.}
    \end{equation}
    Indeed, establishing~\eqref{eq:eps-1-2} would imply that almost surely, for any accumulation point $\msp$ of measures $\bar{\msp}_n$, one has
    \[
     \forall \eps_1>0  \quad \exists N: \quad  \msp(\DD_{\psi;N}) \ge 1-\eps_1,
    \]
     which implies the desired
     \[
     \msp(\DD_{\psi})= \lim_{N\to \infty} \msp(\DD_{\psi;N}) =1.
     \]

     Now, given an initial point $x_0$, let us consider its full $\psi$-orbit
     \[
     \mO_{x_0}=\bigcup_{k,l} \, \{x\mid \psi^l(x)=\psi^k(x_0)\}.
     \]
     Assume first that $x_0$ is not a $\psi$-preperiodic point. Then the full $\psi$-orbit of $x_0$ is an infinite binary tree: every point $x$ is joined to its $\psi$-image and its two $\psi$-preimages. Consider also the forward iterations $x_{(k)}:=\psi^k(x_0)$; we refer to the path
     $x_{0}, \psi(x_0),\psi^2(x_0),\dots$, formed by them, as the \emph{spine} of this tree; see Figure~\ref{fig:spine}.

    \begin{figure}[h!]
        \centering
        \includegraphics[width=0.8\linewidth]{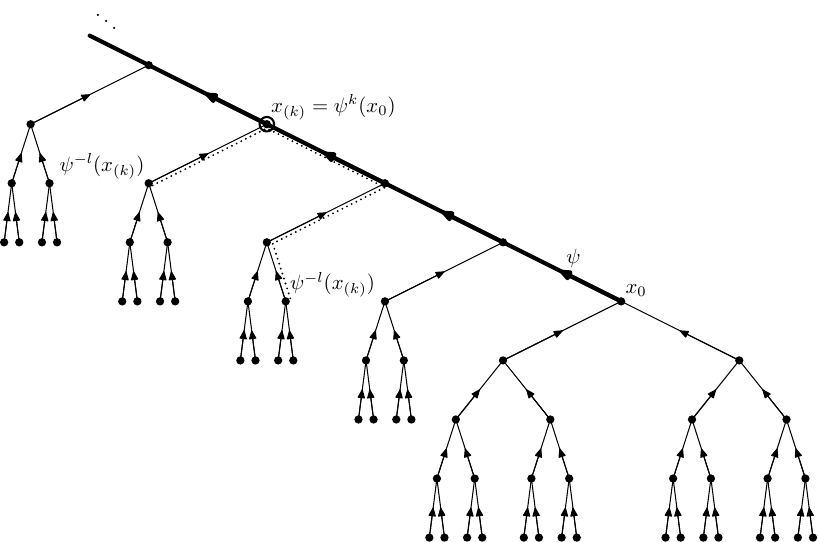}
        \caption{Orbit $\mO_{x_0}$ of a generic point $x_0$ as a binary tree. In bold is shown the \emph{spine}, that is, $\psi$-forward orbit of~$x_0$. The part of the tree below a given point~$x_{(k)}$ of the spine consists of all its $\psi$-preimages (two such preimages are shown).}
        \label{fig:spine}
    \end{figure}

     Now, for a point $y\in \mO_{x_0}$, let $k(y), l(y)$ be the smallest nonnegative integers $(k,l)$ such that $y$ can be represented as
     \[
        y=\psi^{-l}\circ \psi^k(x_0)
     \]
     for some choice of branch of a preimage. In other words, when one joins $x_0$ and $y$ by the shortest path in the tree on Fig.~\ref{fig:spine}, this path first goes upwards from $x_0$ to $\psi^{k(y)} (x_0)$, and then descends to $y=\psi^{-l(y)}(\psi^{k(y)} (x_0))$. We refer to $k(y)$ as the \emph{level} of the point~$y$. In analogy to $G(\cdot)$ considered earlier, define $\mathcal G_{x_0}(y):=l(y)-k(y)$.

     Now, for a sequence of random iterates $x_n=x_n(\omega)$, recurrently defined by $x_{n+1}=f_{\omega_n}(x_{n})$, consider the sequence $k_n=k_n(\omega)$ of levels, reached in the first $n$ iterates: let
     \[
        k_n:=\max_{j\le n} k(x_j)
     \]
     be the maximal level reached in the first $n$ steps. Also, define $l'_{n}=l'_{n}(\omega)$ by
    \[
    l'_{n}:=l(x_n) + k_n-k(x_n)=k_n+\mathcal
    G_{x_0}(x_n),
    \]
    the distance from the point $x_{(k_n)}$ on the spine to $x_n$; we call it the \emph{depth} at the moment~$n$.

    We also let
     \[
        \tau_{k}(\omega):=\min\{n \mid k(x_n)\ge k\}, \quad
        T_{k}(\omega):=\tau_{k+1}-\tau_{k}, \quad k=0,1,\dots
     \]
     be the first moment when the level $k$ is reached and the time spent after such a moment until reaching the level $k+1$, respectively. Note that $T_k$ may take zero values if at a given moment~$n$ the application of $f_{\omega_n}$ moves the point $x_n$ several levels up in one iteration. Also $\tau_k$ might \emph{a priori} take an infinite value (if a random trajectory never reaches level $k$ or above), but as we will see in Lemma~\ref{l:tailz} below, almost surely $\tau_k < \infty$ for all $k$.

     Finally, consider the common partition of the circle into intervals $J_i$, bounded by elements of $\DD_{\psi}$, on each of which each of the maps $f_j$ has the form
     \[
        f_j |_{J_{i}} = \psi^{-l_{ij}} \circ \psi^{k_{ij}}
     \]
     for some choice of branches of the preimages, and let
     \[
     K:=\max_{i,j} k_{ij}, \quad L:=\max_{i,j} l_{ij}.
     \]
     Then, the following statements hold:
     \begin{lemma}\label{l:statements}
            In one iteration, the level cannot be increased by more than $K$:
            \begin{equation}\label{eq:k-increase}
                k_{n+1}-k_n\le K.
            \end{equation}
            The depth at the first moment when a level $k$ is reached admits a uniform upper bound: for every $k$, one has $$l'_{\tau_k}\le L.$$ Furthermore, if for some $n$
            one has $x_n\in J_i$, $\omega_n=j$, then one has
            \begin{equation}\label{eq:mG-change}
            \mathcal G_{x_0}(x_{n+1})=\mathcal G_{x_0}(x_n) -m_{ij},
            \end{equation}
            where $m_{ij}=k_{ij}-l_{ij}$. In particular, if one has $k_n=k_{n+1}$, then the depth changes by the same amount as in~\eqref{eq:mG-change}, namely $l'_{n+1}=l'_n -m_{ij}$.
     \end{lemma}

     We will say that a level $k$ is \emph{$\eps_2$-good} if $x_{(k)}\in U_{\eps_2}(0)$; note that this property depends only on the choice of~$x_0$, and thus is not random. As the map $\psi$ is non-strictly expanding, taking preimages one gets that if $k$ is $\eps_2$-good, then for any $l\le N$ one has
     \[
        \psi^{-l}(x_{(k)}) \in U_{\eps_2}(\DD_{\psi;N}),
     \]
     for every choice of branch of the preimage.
     Hence, all iterations $x_n$ that are of depth at most $N$ and for which the corresponding level $k_n$ is $\eps_2$-good belong to $U_{\eps_2}(\DD_{\psi;N})$.
     Within $t$ steps after reaching a level~$k$, the sequence $x_n$ cannot reach a depth~$l'$ exceeding $L \cdot (t+1)$. Hence, after reaching an $\eps_2$-good level $k$ one needs at least $\lfloor \frac{N}{L} \rfloor-1$ steps for the trajectory to leave $U_{\eps_2}(\DD_{\psi;N})$ while staying on level~$k$. Thus, one has
     \begin{multline}\label{eq:T-sum}
        1-\bar{\msp}_n(U_{\eps_2}(\DD_{\psi;N})) = \frac{1}{n} \# \{j\le n \mid x_j \notin U_{\eps_2}(\DD_{\psi;N})\}
        \\
        \le \frac{1}{n}\left(\sum_{k\le k_n, \atop \text{$k$ is $\eps_2$-good}} \max(T_k-(\lfloor \tfrac{N}{L}\rfloor-1),0) + \sum_{k\le k_n, \atop \text{$k$ is not $\eps_2$-good}} T_k  \right)
     \end{multline}

     Now, note that $(\tau_k(\omega),x_{\tau_k(\omega)})$ is a Markov process; consider the conditional distribution of {$T_{k}(\omega)$} with respect to~$x_{\tau_k}$. The same drift arguments that we have already seen in Lemmata~\ref{l:applying-FL} and~\ref{l:exptail} imply a uniform exponential bound on its tails. Namely, we have the following lemma.
     \begin{lemma}\label{l:tailz}
         There exist constants $C, \beta>0$ such that for every $t,k$ and any $x'$ of level $k$, the conditional probability satisfies
         \[
            \Prob(T_{k}(\omega)>t \mid x_{\tau_k}=x') < C e^{-\beta t}.
         \]
                In particular, the time $T_k$ spent on each level~$k$ and the moment~$\tau_k$ of first entering such a level are finite almost surely.
     \end{lemma}
    Postponing its proof for the moment, let us use it to conclude the proof of Proposition~\ref{prop:acc-pts}.
    Consider the random variables that appear as summands in the right-hand side of~\eqref{eq:T-sum}: let
    \[
        \eta_k:=\max(T_k-(\lfloor \tfrac{N}{L}\rfloor-1),0) \cdot \Ind_{\text{$k$ is $\eps_2$-good}}, \quad
        \widetilde{\eta}_k:= T_k \cdot \Ind_{\text{$k$ is not $\eps_2$-good}}.
    \]
    Then the right-hand side of~\eqref{eq:T-sum} is equal to $\frac{1}{n}(\Eta_{k_n}+\wEta_{k_n})$, where
    \begin{equation}\label{eq:sum-E-T-eta}
    \Eta_{k'}:=\sum_{k\le k'} \eta_k, \qquad  \wEta_{k'}:=\sum_{k\le k'} \weta_k.
    \end{equation}
    {Now, consider the family of $\sigma$-algebras $\mF_{\tau_k}$ generated by the process $(x_n)_{n \ge 0}$ up to the stopping time $\tau_k$; then, all $\eta_k$, $\weta_k$ are measurable w.r.t.\ $\mF_{\tau_{k'}}$ once $k'> k$.} We then have the following uniform upper bound for the conditional expectations.
    \begin{lemma}\label{l:eta-conditional}
        For every $\eps'>0$ there exists $N$ such that one has a uniform upper bound
        \[
        \E(\eta_{k} \mid \mF_{\tau_{k}}) \le \eps'.
        \]
        Also, there exists $C_T$ such that
        \[
        \E(T_{k} \mid \mF_{\tau_{k}}) \le C_T.
        \]
    \end{lemma}
    \begin{proof}
        The uniform estimate for the conditional expectation of $T_k$ follows immediate from Lemma~\ref{l:tailz}: one has
        \begin{equation}\label{eq:T-cond-exp}
            \E(T_k \mid \mF_{\tau_{k}}) = \sum_{t=0}^{\infty} \Prob(T_k > t \mid x_{\tau_{k}}) < \sum_{t=0}^{\infty} C e^{-\beta t} =: C_T.
        \end{equation}
        For the conditional expectation $\E\left(\eta_k \mid \mF_{\tau_{k}}\right)$, one has
        \begin{equation}\label{eq:T-N-cond-exp}
                    \E\left(\max(T_k-(\lfloor \tfrac{N}{L}\rfloor-1),0) \mid \mF_{\tau_{k}}\right) < \sum_{t=\lfloor \frac{N}{L}\rfloor-1}^{\infty} \Prob(T_k > t \mid x_{\tau_{k}}) < \sum_{t=\lfloor\frac{N}{L}\rfloor-1}^{\infty} C e^{-\beta t}.
        \end{equation}
        For any given $\eps'>0$, choosing $N$ sufficiently large, we can ensure that the sum in the right-hand side of~\eqref{eq:T-N-cond-exp} does not exceed $\eps'$.
    \end{proof}

    Now, consider the sum of these conditional expectations: let
    \[
        S_{k'} = \sum_{k\le k'} \E(\eta_{k} \mid {\mF_{\tau_{k}}}),
        \qquad
        \wS_{k'} = \sum_{k\le k'} \E(\weta_{k} \mid {\mF_{\tau_{k}}}).
    \]
    We then have an upper bound for these sums: for every $\eps'>0$, let us fix the corresponding~$N$ provided by Lemma~\ref{l:eta-conditional}. Then
    \begin{equation}\label{eq:S-bound}
        S_{k'} = \sum_{k\le k'} \E(\eta_{k} \mid {\mF_{\tau_{k}}}) \le k' \eps'.
    \end{equation}
    At the same time, as the point~$x_0$ satisfies the condition~\eqref{eq:conv-0}, we have
    \[
        \frac{1}{k'} \# \{k\le k' \mid \text{$k$ is not $\eps_2$-good}\} \to 0, \quad k'\to \infty.
    \]
    Combining this with the second conclusion of Lemma~\ref{l:eta-conditional}, we get
    \[
        \wS_{k'} \le C_T \cdot \# \{k\le k' \mid \text{$k$ is not $\eps_2$-good}\},
    \]
    hence, for all sufficiently large $k'$, we also have the non-random upper bound $\wS_{k'}\le \eps'k'$.

    Extending the estimates for the sums of conditional expectations $S_k$, $\wS_k$ to those for the sums $\Eta_k$, $\wEta_k$ is done using standard large-deviation-type arguments. Namely, consider the differences $\Eta_{k'}-S_{k'}$ and $\wEta_{k'}-\wS_{k'}$. Note that both these processes are martingales (with respect to the $\sigma$-algebras $\mF_{\tau_{k'+1}}$). Indeed, each $\eta_k$ is measurable w.r.t.\ $\mF_{\tau_{k+1}}$; moreover,
    \[
        \Eta_{k'}-S_{k'} = (\Eta_{k'-1}-S_{k'-1}) + \left(\eta_{k'} - \E(\eta_{k'}\mid \mF_{\tau_{k'}})\right),
    \]
    and the conditional expectation of the second summand w.r.t.\ $\mF_{\tau_{k'}}$ vanishes.

    Finally, the exponential moments of the conditional distributions of increments $\eta_k=\Eta_k-\Eta_{k-1}$ and $\weta_k=\wEta_k-\wEta_{k-1}$ w.r.t.~$\mF_{\tau_k}$ are uniformly bounded due to Lemma~\ref{l:tailz}. The same statement thus holds for the increments of the martingales $\Eta_k-S_k$ and $\wEta_k-\wS_k$ that are conditionally centred versions of $\eta_k$ and $\weta_k$ respectively. Thus, a martingale version of Bernstein's inequality {(see for example \cite[Theorem 1.2A]{MR1681153})} provides a large-deviation-type bound: for every $\eps'>0$ there exist $c_{\Eta}, c_{\wEta}>0$ such that for all sufficiently large $k'$,
    \[
        \Prob(\Eta_{k'} - S_{k'} > \eps' k') < \exp(- c_{\Eta} k'), \quad
        \Prob(\wEta_{k'} - \wS_{k'} > \eps' k') < \exp(- c_{\wEta} k').
    \]
    Applying the Borel--Cantelli lemma, we get that (for every fixed $\eps'$) almost surely for all sufficiently large $k'$ one has
    \[
        \Eta_{k'}+\wEta_{k'} \le 2\eps' k' + S_{k'} + \wS_{k'} \le 4 \eps' k'.
    \]
    Now, one has $k_n<K\cdot n$ by~\eqref{eq:k-increase} in Lemma~\ref{l:statements}; thus, for all sufficiently large $n$, the right-hand side of~\eqref{eq:T-sum} almost surely does not exceed
    \[
        \frac{1}{n} (\Eta_{k_n} + \wEta_{k_n}) \le 4\eps' \frac{k_n}{n} \le 4 K \eps'.
    \]

    Thus, taking $\eps':=\frac{\eps_1}{4K}$ and choosing the corresponding $N$, we obtain the desired~\eqref{eq:eps-1-2}.

     This concludes the proof of Proposition~\ref{prop:acc-pts} for a non-$\psi$-preperiodic point $x_0$.

     Lastly, if $x_0$ is $\psi$-preperiodic, then due to~\eqref{eq:conv-0} one has $\psi^n(x_0)=0$ for some $n$, and thus $x_0\in\DD_{\psi}$. In this case, the conclusion follows from Lemma~\ref{l:exptail}. This concludes the proof of Proposition~\ref{prop:acc-pts}.
\end{proof}

     We have established Proposition~\ref{prop:acc-pts}, together with Theorem~\ref{t:parabolic},
     this concludes the proof of Theorem~\ref{thm:physical}.
     \end{proof}

     \begin{proof}[Proof of Lemma~\ref{l:tailz}]
        The proof largely mimics the proof of Lemma~\ref{l:exptail}.
        Namely, the same estimates as in the proof of Lemma~\ref{l:applying-FL} imply that the function $V_{x_0}(x)=e^{\theta \mathcal G_{x_0}(x)}$ satisfies, for any $x\in \mO_{x_0}$,
        \begin{equation}\label{eq:gamma-V}
            \E V_{x_0}(f(x)) \leq \gamma V_{x_0}(x),
        \end{equation}
        where $\theta>0$ and $\gamma<1$ are the same as in Lemma~\ref{l:applying-FL} (and this choice does not depend on~$k$).
        Consider now the sequence of random variables
        \begin{equation}\label{eq:xi-t-def}
            \xi_t := \gamma^{-t} e^{\theta k} \,  V_{x_0}(x_{\tau_k+t}) \Ind_{T_{k}>t} = \gamma^{-t} e^{\theta l'_{\tau_k+t}}\Ind_{T_{k}>t}.
        \end{equation}
        Then~\eqref{eq:gamma-V} implies that the family $\xi_t$ forms a nonnegative supermartingale (with respect to the natural family of $\sigma$-algebras $\mF_t$, generated by $\omega_0,\dots,\omega_{\tau_k+ t-1}$): either $T_k(\omega)\le t$, and then $\xi_t(\omega)=\xi_{t+1}(\omega)=0$, or
            \[
        \E(\xi_{t+1}\mid \mF_t) \le \gamma^{-(t+1)} \E (V_{x_0}(x_{\tau_k+t+1})\mid \mF_t) \le \gamma^{-(t+1)} \cdot \gamma V_{x_0}(x_{\tau_k+t}) \le \xi_t.
    \]
    In particular, the sequence of expectations $\E \xi_t$ is non-increasing. Now, the initial value satisfies
    \[
        \xi_0=e^{\theta l'_{\tau_k}} \le e^{\theta L}.
    \]
    On the other hand, if $T_k> t$, then the indicator in the right-hand side of~\eqref{eq:xi-t-def} is equal to~$1$, and as the depth $l'_{\tau_k+t}$ is nonnegative, we have
    \[
        \xi_t =\gamma^{-t} e^{\theta l'_{\tau_k+t}} \ge \gamma^{-t}.
    \]
    Since the bound $\xi_0 \le e^{\theta L}$ holds uniformly in $x_{\tau_k}$
    (by Lemma~\ref{l:statements}), Markov's inequality gives the desired
    conditional estimate
    \[
        \Prob(T_k > t \mid x_{\tau_k} = x')
        \le \frac{\E(\xi_t \mid x_{\tau_k} = x')}{\gamma^{-t}}
        \le \gamma^t \cdot \E(\xi_0 \mid x_{\tau_k} = x')
        \le e^{\theta L} \cdot \gamma^t.
    \]
\end{proof}

    \begin{proof}[Proof of Theorem~\ref{thm:circle}]
    Take any random dynamical system~$(f_{1,(0)},\dots,f_{s,(0)};p_1,\dots,p_s)$, consisting of maps $f_{i,(0)}$ from the Thompson semigroup $\TT^+$, that satisfies the assumptions of Theorem~\ref{thm:atomic}, contains the doubling map $\varphi$, and has no common invariant measure. One such example is provided by the piecewise-affine system from Example~\ref{ex:PL}; note that the absence of a common invariant measure is guaranteed by Lemma~\ref{lem:nocommon}.

    Now, take any degree~$2$ map $\psi:\Sc\to\Sc$ of class $C^{\infty}$ such that
    \begin{itemize}
        \item $\psi$ is $C^{\infty}$-tangent to the identity at its fixed point~$0$,
        \item $\psi'(x)>1$ at every point $x\neq 0$,
    \end{itemize}
    and let $h$ be the conjugacy between $\psi$ and $\varphi$ (cf.~\eqref{eq:psi-phi-h}). Take the maps
    \[
        f_i=h^{-1}\circ f_{i,(0)}\circ h
    \]
    and consider the random dynamical system (cf.~\eqref{eq:conjugates})
    \[
        (f_1,\dots,f_s;p_1,\dots,p_s).
    \]
    For the particular case of Example~\ref{ex:PL}, this provides the system that appears in Example~\ref{ex:conju} (with the graphs shown in Fig.~\ref{fig:GS-example}).
    Then, by Corollary~\ref{cor:smooth}, the constructed system~\eqref{eq:conjugates} consists of $C^{\infty}$-smooth locally invertible maps, and admits an atomic stationary measure~$\msp$, supported on the (countable) set~$\DD_{\psi}$. Theorem~\ref{thm:physical} (as $\varphi$ is one of the $f_{i,(0)}$) then states that the measure~$\msp$ is the unique stationary measure supported on~$\DD_{\psi}$, and that this measure is a physical measure for this system. This concludes the proof of Theorem~\ref{thm:circle}.
    \end{proof}

\begin{proof}[Sketch of the proof of Theorem~\ref{t:parabolic}]
    Recall that $\psi$ is conjugate to the doubling map $\varphi:x \mapsto 2x \mod 1,$ that is, $h \circ \psi  =  \varphi\circ h$ for some circle homeomorphism~$h$. Consider the interval $J=[a,b]:=h^{-1}([\frac{1}{3},\frac{2}{3}])$, bounded by a period-$2$ orbit of~$\psi$, and a point $c=h^{-1}(\frac{1}{2})\in J$: for this point, $\psi(c)=0$.

    \begin{figure}[ht!]
        \centering
        \includegraphics[height=3cm]{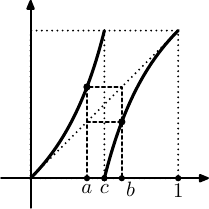}
\\[7mm]
        \includegraphics[width=7cm]{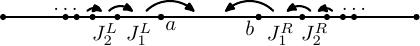}
        \caption{Interval $J$ and the first return map on it}
        \label{fig:first-return}
    \end{figure}

    Now, consider the first return map on $J$. To describe it, note that $\psi([a,c])=[b,1]$ and that $\psi([c,b])=[0,a]$. Now, the interval $J$ is a fundamental domain for both branches of $\psi_0$, viewed as a repelling map in the neighbourhood of~$0$; hence, the interval $[0,a]$ is decomposed into a chain of adjacent preimages,
    \[
    [0,a]=\bigcup_{j\geq 1} J^L_j, \, \text{ where } J^L_j :=\psi_0^{-j}(J).
    \]
    Similarly,
    \[
    [b,1]=\bigcup_{j\geq1} J^R_j, \, \text{ where } J^R_j :=\psi_1^{-j}(J).
    \]
    We thus get the decomposition of $[a,b]$ into intervals $\psi_1^{-1}(J^L_j)$ and $\psi_0^{-1}(J^R_j)$, each of which is mapped by the corresponding power of $\psi$ onto~$J$. Finally, the first return map $\Phi$ on $J$ is given by the union of the maps
    \[
        f^L_j:= \psi^{j+1}|_{\psi_0^{-1}(J^L_j)}, \quad f^R_j:= \psi^{j+1}|_{\psi_1^{-1}(J^R_j)}.
    \]

    Now, as $\psi$ is uniformly expanding on $J$, the return branches satisfy a uniform bound on the distortion coefficient: there exists $C_{\kappa}$ such that the distortion
    \[
        c_{\kappa}(f) := \left\|\frac{D^2f(x)}{(Df(x))^2}\right\| = \|D \log D f^{-1}\|
    \]
    does not exceed $C_{\kappa}$ for all $f=f^L_j, f^R_j$, $j=0,1,\dots$.

    These properties suffice to conclude that the first return map $\Phi$ has an absolutely continuous invariant measure with a positive continuous density. This is done by the standard technique of considering the iterations of the Lebesgue measure and controlling the ratio of its densities at two different points; the distortion estimates imply that this ratio stays bounded. This measure can also be shown to be ergodic; for Lebesgue-a.e.\ initial point $x_0\in J$, the iterates $\Phi^m(x_0)$ are distributed with respect to this measure.

    Finally, one notices that the first return time has a non-integrable singularity (of type at least $\frac{1}{x}$) at the point~$c$. Therefore, the Birkhoff averages of the return time tend to infinity for Lebesgue-a.e.\ initial point, and hence the $\psi$-iterates spend an asymptotically full proportion of the time near the neutral fixed point~$0$, giving~\eqref{eq:conv-0}.
\end{proof}

\subsection{Rotation-symmetrized example: description of stationary measures}\label{ss:concluding}
There is a modification of Example~\ref{ex:PL} in which the stationary measure $\msp$ admits a transparent description in terms of $\varphi$-invariant measures. Let $f_1,f_2$ be the maps of Example \ref{ex:PL}, and $R_{\alpha}$ be the rotation by~$\alpha$, that is, $R_{\alpha}:x\mapsto x+\alpha \mod \Z$. Consider the system obtained by replacing each $f_i$ with all of its conjugates $R_{-a/4}\,f_i\,R_{b/4}$ for $a,b\in\{0,1,2,3\}$ independent and uniform. This yields the random dynamical system
\begin{equation}\label{eq:f-ab}
\bigl(R_{-a/4}\,f_i\,R_{b/4}\,;\,\tfrac{p_i}{16}\bigr)_{i\in\{1,2\},\;a,b\in\{0,1,2,3\}}
\end{equation}
on~$\Sc$. Then, the Markov chain associated to this system satisfies the following two $R_{1/4}$-invariance properties:
\begin{itemize}
    \item For any $x,x'\in\Sc$, the probabilities that the Markov chain goes in one step from $x$ to each of $x'$, $R_{1/4}x'$, $R_{1/2}x'$, $R_{3/4}x'$ are all equal.
    \item For any $x,x'\in\Sc$, the probabilities that the Markov chain goes in one step to $x'$ from each of $x$, $R_{1/4}x$, $R_{1/2}x$, $R_{3/4}x$ are all equal.
\end{itemize}
This $R_{1/4}$-invariance motivates the consideration of the image under the quotient by this symmetry: let $y_n=\varphi_j^2 (x_n)$. Then, the random process $(y_n)$ is also a Markov chain. To find its transition probabilities, note that conditionally to $y_n=y$, all the four $\varphi^{2}$-preimages of $y$ are equiprobable as~$x_n$; thus (considering the maps $f_1$, $f_2$ on each of the four intervals $[\frac{a}{4},\frac{a+1}{4}), \, a=0,1,2,3$; compare with Fig.~\ref{fig:PL-example}), we get
\begin{equation}\label{eq:MC}
    y_{n+1} = \begin{cases}
    \varphi(y_{n}) & \text{with probability $p_+$}, \\
    y_{n} & \text{with probability $p_0$},\\
    \varphi_0^{-1}(y_{n}) & \text{with probability $\frac{1}{2}p_-$},\\
    \varphi_1^{-1}(y_{n}) & \text{with probability $\frac{1}{2}p_-$},
\end{cases}
\end{equation}
where $p_+,p_0,p_-$ are the probabilities that are given by
\[
p_+ = p_1 + \frac{p_2}{4}, \quad p_- = \frac{p_2}{2}, \quad p_0 = \frac{p_2}{4},
\]
and $\varphi_0^{-1}$ and $\varphi_1^{-1}$ are two branches of $\varphi^{-1}$, corresponding to
\[
\varphi_0:=\varphi|_{[0,\frac{1}{2})}, \quad
\varphi_1:=\varphi|_{[\frac{1}{2},1)}.
\]
(Note that~\eqref{eq:MC} is very close to the systems studied in \cite{MR4975394}.)

Now, the stationary measures of the random dynamical system~\eqref{eq:f-ab} are in bijective correspondence with those of the Markov chain~\eqref{eq:MC}. Namely, to a stationary measure~$\msp$ for~\eqref{eq:f-ab} corresponds a stationary measure $\msp'=\varphi^2_* \msp$ for~\eqref{eq:MC}. On the other hand, to a stationary measure $\msp'$ for~\eqref{eq:MC} corresponds a stationary measure $\msp=Q^2 \msp'$, where
\begin{equation}\label{eq:Q-def}
Q := \tfrac12\bigl((\varphi_0^{-1})_* + (\varphi_1^{-1})_*\bigr)
\end{equation}
is the operator, acting on probability measures on the quotient circle, that corresponds to taking each of two $\varphi$-preimages equiprobably.

Next, the system~\eqref{eq:f-ab} is expanding in average if an only if~$p_+>p_-$. Now, Markov chain~\eqref{eq:MC} also preserves the set of dyadic rationals~$\DD$, and as the chain applies $\varphi_0^{-1}$ and $\varphi_1^{-1}$ equiprobably, its stationary measure $\msp'$ should be equidistributed on every distance $n$ set $R_j:=\{y\in \DD \mid \dT(y)=j\}$ (the measure $\msp'$ is unique due to the positive recurrence, and thus the group of the automorphisms of the rooted tree should preserve it).
Hence, describing its stationary measure $\msp'$ is equivalent to finding a sequence of masses $m_j=\msp'(R_n)$, as then one has $\msp'(y)=\frac{m_j}{\# R_j}$ for any $y\in R_j$. These masses are the distribution of the push-forward image $\msp''=\dT_* \msp'$, that is a stationary measure for the quotient Markov chain $j_n=\dT(y_n)$. The latter is the Markov chain on $\{0,1,2,\dots\}$ with the transition  probabilities (see Fig.~\ref{fig:MC-line})
\begin{equation}\label{eq:MC-line}
    j_{n+1} = \begin{cases}
    j_{n}-1 & \text{with probability $p_+$}, \\
    j_{n} & \text{with probability $p_0$},\\
    j_{n}+1 & \text{with probability $p_-$},
    \end{cases}
    \quad j_n>0,
\end{equation}
with the natural modification at $j_n=0$: one jumps away from it to~$j_{n+1}=1$ with the probability~$\frac{p_-}{2}$, otherwise staying at the state~$0$.

\begin{figure}
    \centering
    \includegraphics[width=0.5\linewidth]{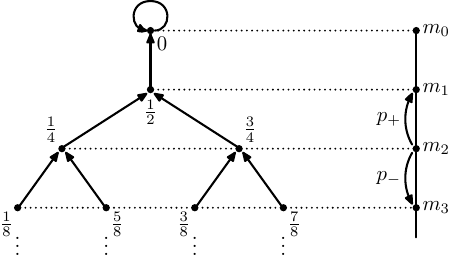}
    \caption{Tree of iterations of $\varphi$ and the Markov chain and the stationary measure for the distances $\dT(y_n)$}
    \label{fig:MC-line}
\end{figure}

The stationarity condition for the Markov chain~\eqref{eq:MC-line} is given by the relations
\begin{eqnarray}
    m_j = p_+ m_{j+1} + p_- m_{j-1} + p_0 m_j,  \quad  j=2,3\dots,  \label{eq:relation}
    \\
    m_1 = p_+ m_2 + \frac{p_-}{2}\, m_0 + p_0 m_1;
    \label{eq:relation-1}
    \\
    m_0= p_+ m_{1} + \left(p_+ +\frac{p_-}{2} + p_0\right) m_{0}
    \label{eq:relation-0}
\end{eqnarray}
note that the relation~\eqref{eq:relation-0} follows from~\eqref{eq:relation} and~\eqref{eq:relation-1}: the total mass is conserved, thus the stationarity in all points but one implies the stationarity for the last one, too.

Now, the characteristic equation $1=p_+ r + p_- r^{-1} + p_0 $ for the recurrence relation~\eqref{eq:relation} has roots $r=1$ (that corresponds to the constant sequence and not to the probability measure) and $r=\frac{p_-}{p_+}<1$. Thus, the stationary measure for the Markov chain~\eqref{eq:MC-line} is given by
\begin{equation}\label{eq:m-j}
m_0 = 2c, \qquad
m_j = c\, r^j, \quad j= 1,2,\dots
\end{equation}
with the normalization constant $c = \tfrac{1-r}{2-r}$.

The equidistribution of $\msp'$ on the level sets together with~\eqref{eq:m-j}
gives the single formula
\begin{equation}\label{eq:nu-prime}
\msp'(\{y\})
= 2c \left(\frac{r}{2}\right)^{\dT(y)}, \qquad y\in\DD,
\end{equation}

It remains to lift $\msp'$ back to a stationary measure on the original circle. As we have already seen, it is given by $\msp=Q^2 \msp'$ (equidistributing the weight of any given point $y\in\DD$ between its four $\varphi^{-2}$-preimages).
As this procedure is invariant under the conjugacy, we thus obtain an explicit description for the atomic physical stationary measure for the system, built from~\eqref{eq:f-ab} using the function $\psi$ as in Theorem~\ref{thm:physical} (note that this conjugated system can also be written explicitly in terms of compositions of $\psi$ and branches $\psi_0^{-1},\psi_1^{-1}$ it in the same way as in Example~\ref{ex:conju}).

It is interesting to note that the explicit description above can be seen as a particular case of a more general construction. Namely, assume that we are given any $\varphi$-invariant measure~$\msp_{(0)}$ (the description above would correspond to~$\msp_{(0)}=\delta_0$). The following statement allows to construct from it a stationary measure for the Markov chain~\eqref{eq:MC} and hence for the random dynamical system~\eqref{eq:f-ab}.
\begin{proposition}\label{p:new-stationary}
    Let $\msp_{(0)}$ be any invariant measure for the doubling map~$\varphi$. Then the measure
    \begin{equation}\label{eq:nu-Q}
        \msp':= c' \sum_{j=0}^{\infty} (rQ)^j \msp_{(0)}, \quad r=\frac{p_-}{p_+}, \quad c'=1-r,
    \end{equation}
    is a stationary measure for the Markov chain~\eqref{eq:MC}, and $\msp=Q^2\msp'$ is thus a stationary measure for~\eqref{eq:f-ab}.
\end{proposition}
\begin{proof}
    The one-step transition operator for the Markov chain~\eqref{eq:MC} is $p_+\varphi_* + p_0\cdot \id + p_- Q$. Applying it to~\eqref{eq:nu-Q}, using $\varphi_* Q=\id$ and $\varphi_* \msp_{(0)}=\msp_{(0)}$ and grouping together terms with the same power of~$Q$ completes the proof.
\end{proof}
As it was already mentioned, one can see from a straightforward computation that the description~\eqref{eq:m-j}, \eqref{eq:nu-prime} above is a particular case of Proposition~\ref{p:new-stationary}, applied to the Dirac measure $\msp_{(0)}=\delta_0$ as an invariant measure of~$\varphi$.

\section{Constructions of examples}\label{s:examples-proofs}
In this section we discuss the constructions and properties of Examples~\ref{ex:torus}--\ref{ex:sphere2}.

\begin{proof}[Construction of Example~\ref{ex:torus}]
    For $\eps>0$ small enough, each $F_j$ is a diffeomorphism onto
    its image with $F_j(M)\subset\mathrm{int}\,M$, and the projection
    $\pi(x,\cwnew)=x$ satisfies $\pi\circ F_j=f_j\circ\pi$. The system is
    thus a random contracting skew product over $(f_j;p_j)$, a variant of
    the classical Smale
    solenoid~\cite[Ch.~17.1]{katok_introduction_1995}.
    Next, it is known that for the contracting random skew products for every stationary measure~$\msp$ in the base there exists a unique stationary measure~$\hmsp$ of the skew product. Indeed, consider the Perron-Frobenius operator $P$ of the skew product. Its action preserves the space of measures that project to $\msp$ (due to its stationarity). Now, when restricted to this space, the operator $P$ contracts in the fiberwise transport metric (integrated with respect to~$\msp$ in the base); this implies the uniqueness of its fixed point~$\hmsp$.

    For Lebesgue-a.e.\ point $b_0=(x_0,z_0)\in M$, the distribution of the first coordinate of its trajectory almost surely converges to~$\msp$ due to the choice of the maps~$f_i$ and Theorem~\ref{thm:circle}. On the other hand, for the random trajectory of the point $b_0$ almost surely all the accumulation points of its distribution are stationary measures of the system. As the stationary measure~$\hmsp$ is unique in the class of measures on~$M$ that project to~$\msp$, these distributions for a Lebesgue-generic point $b_0$ thus converge to~$\hmsp$, proving~\eqref{eq:convergence-hmsp}. Hence, the measure $\hmsp$ is the physical measure of the system.

    Finally, any common $F_j$-invariant probability on $M$ would push forward
    under $\pi$ to a common $f_j$-invariant probability on $\Sc$,
    contradicting Theorem~\ref{thm:circle}; the same theorem forbids
    a common $f_j$-invariant measure on
    $\{\D_a\}_{a\in\Sc}= \Sc$.
\end{proof}

\begin{proof}[Construction of Example~\ref{ex:sphere}]

    Let the solid torus $M=\Sc\times\D$ be one of the two tori in the standard Heegaard splitting of the 3-sphere:
    \[
        \Sph = M\cup M', \qquad M'=\D\times\Sc, \qquad
        M\cap M' = \partial M = \partial M'=\Sc\times \Sc.
    \]
    We claim that there is a modification $\hF_j:M\to M$ of the maps $F_j$ from Example~\ref{ex:torus} that is isotopic to the identity in the class of maps $\hF:M\to \Sph$ that are diffeomorphisms onto their images.
    Once this is done, by the isotopy extension theorem (see for
    example~\cite[Theorem~8.1.3]{MR448362}), each~$\hF_j$ extends to a diffeomorphism $\tF_j:\Sph\to\Sph$.

    It is noted in Katok and Hasselblatt~\cite[Chapter 17.1]{katok_introduction_1995} that the solenoid map $F_j$, given by~\eqref{eq:F-j-def}, cannot be extended to a diffeomorphism of the sphere~$\Sph$, if the corresponding $f_j$ is of degree~$\dg_j>1$.

    One way to see it is that the parallels on the solid torus (that is, circles $\{z=c \in \D\}$) are unlinked, but become linked after the application of~$F_j$. A manifestation of this effect can be seen in the following way. Imagine that one puts a thick rope or a cord circled many times, and then lifts it and straightens by holding both ends and pulling them away from each other. Then, the cord becomes twisted.

    So in order to have a map that can be extended to the sphere (and thus necessarily that does not link the parallels), one has to ``untwist'' the map $F_j$. This is done in the following lemma (in fact, this construction is a solution to~\cite[Exercise~17.1.4]{katok_introduction_1995}).

    \begin{lemma}\label{l:untwist}
        For every smooth monotone circle map $f:\Sc\to\Sc$ of topological degree $\dg>0$, for every $\eps>0$ sufficiently small, the map
        \begin{equation}\label{eq:hF-def}
            \hF:M\to \Sph, \quad \hF(x,\cwnew) = \left(f(x), \frac{1}{2} e^{2\pi i x}+ e^{-2\pi i (\dg-1)x} \cdot \eps \cwnew\right)
        \end{equation}
        is isotopic to the identity
        in the class of diffeomorphisms onto their images.
    \end{lemma}

    \begin{proof}
        Note first that any isotopy between monotone circle maps $f$ extends for a sufficiently small $\eps>0$ (by plugging this isotopy directly into the definition~\eqref{eq:hF-def}) to an isotopy between the corresponding maps~$\hF$. Thus we can, without loss of generality, replace the map~$f$ by the map $x\mapsto \dg x$ in the same isotopy class.

       Now, to be precise, let us fix an embedding of $M$ in $\Sph$. Namely, let the 3-sphere be defined as $\Sph=\{(\cznew,\cwnew)\in \C^2\mid |\cznew|^2+|\cwnew|^2= 1\}$. Consider the action of~$\C^*$ on~$\C^2$ by $(1,\dg)$-weighted scaling:
        \begin{equation}\label{eq:g-rho}
            g_{\rho}(\cznew,\cwnew)=(\rho \cznew, \rho^\dg \cwnew), \quad \rho\in \C^*.
        \end{equation}
        Then, for the action of the subgroup of positive real numbers $\R_+^*\subset \C^*$, the orbit of any point of $\C^2\setminus \{(0,0)\}$ intersects the sphere in a unique point, thus providing a projection  $\proj:\C^2\setminus \{(0,0)\}\to\Sph$. Now, we embed $M\subset \Sc\times \C$ to $\Sph$ by a composition of a natural ``vertical'' embedding
        \[
        \iota:\Sc\times \C \to \C^2\setminus \{(0,0)\}, \quad (x,\cwnew)\mapsto (e^{2\pi i x},\cwnew)
        \]
        and of the projection $\proj$: the restriction $\proj|_{\iota(\Sc\times \C)}:\Sc\times \C \to \Sph\setminus \{\cwnew=0\}$ is a diffeomorphism.

        Next, consider linear rescalings
        \[
        L_a:(x,\cwnew)\mapsto (x,a\cwnew), \quad a\in \C^*,
        \]
        acting on $\Sc\times \C$. These rescalings are isotopic to the identity. In terms of these rescalings, the map~\eqref{eq:hF-def} with $f(x)=\dg x$, considered as a map from $\Sc\times \C$ to itself (that is diffeomorphism onto its image when restricted to~$M$) can be seen as a composition $L_{\frac{1}{2}}\circ \hF'\circ L_{2\eps}$, where
        \begin{equation}\label{eq:F-hat-p}
            \hF'(x,\cwnew)= (\dg x, e^{-2\pi i (\dg-1)x} \cwnew + e^{2\pi i x}) = \left(\dg x, e^{2\pi i x} \cdot (1+  e^{-2\pi i \dg x} \cwnew ) \right);
        \end{equation}
        note that the map~\eqref{eq:F-hat-p} restricts, for a sufficiently small $\eps>0$, to a diffeomorphism from \(M_{2\eps}:=L_{2\eps}(M)=\{|\cwnew| \leq 2\eps\}\) onto its image.

        As both linear rescalings $L_{2\eps},L_{\frac{1}{2}}$ are isotopic to the identity in $\Sc\times \C$, it thus suffices to consider the map~$\hF'$, given by~\eqref{eq:F-hat-p},
        in an (arbitrarily small) $2\eps$-neighbourhood $M_{2\eps}=\{(x,\cwnew)\mid |\cwnew|\le 2\eps\} \subset \Sc\times \C$ of the central circle $\{\cwnew=0\}$.  Namely, it suffices to show that the maps $\proj\circ \iota|_{M_{2\eps}}$ and $\proj \circ \iota \circ \hF'|_{M_{2\eps}}$ are isotopic in the class of maps from $M_{2\eps}$ to $\Sph$ that are diffeomorphisms onto their images. Note that the composition $\iota \circ \hF'$ can be also written as $\Psi_0'\circ \iota$, where
        \[
            \Psi_0'(\cznew,\cwnew)= (\cznew^{\dg}, \cznew \cdot (1+\cznew^{-\dg}\cwnew)),
        \]
        thus we are interested in an isotopy between the maps
        \[
        \proj|_{\iota(M_{2\eps})} \quad \text{and}  \quad  \proj\circ \Psi_0'|_{\iota(M_{2\eps})}.
        \]

        Let us interchange the coordinates of the map $\Psi_0'$. Namely, consider
        the map $\Psi_0: \C^2\to\C^2,$
        \begin{equation}\label{eq:F-hat-pp}
            \Psi_0:(\cznew,\cwnew)\mapsto
            (\cznew\cdot (1+\cwnew\cznew^{-\dg}),\cznew^{\dg}),
        \end{equation}
        as well as the interchanging of coordinates map $R:(\cznew,\cwnew)\mapsto (\cwnew,\cznew)$. Then, $\Psi_0 = R\circ\Psi_0'$, and the map $R$ that is a rotation by~$\pi$ around the $(1,1)$-direction in $\C^2$ (and thus~$R$ is isotopic to the identity in the group of isometries of $\C^2$). We claim that this implies an existence of an isotopy between
        \begin{equation}\label{eq:Psi-0-prime}
            \proj\circ \Psi_0'|_{\iota(M_{2\eps})}\quad \text{and} \quad \proj\circ \Psi_0|_{\iota(M_{2\eps})} = \proj\circ R\circ \Psi_0'|_{\iota(M_{2\eps})}
        \end{equation}
        Indeed, it would be tempting to insert the isotopy $R_t$ between $R$ and $\id$ directly in~\eqref{eq:Psi-0-prime}, but it could be that for an intermediate rotation post-composition with the projection will lead to a degeneracy or non-injectivity and therefore failing to be an isotopy. Thus, instead we first rescale $\Psi_0'$ vertically,
        passing from it to its composition with $L'_{\delta}:(\cznew,\cwnew)\mapsto (\cznew,\delta\cwnew)$
        for a sufficiently small $\delta>0$. Then, the resulting image is sufficiently close to the central circle $\{\cwnew=0\}$ on the cylinder $\{|\cznew|=1\}$, and the transversality of the projection trajectories to the sphere $\Sph$ implies that the composition with the projection stays a diffeomorphism on the image along all the isotopy $\proj \circ R_t \circ L'_\delta \circ\Psi_0'$. Finally, we do a linear rescaling for the $\cznew$ coordinate, dividing it by $\delta$ and thus arriving to the desired map $\proj \circ \Psi_0:\iota(M_{2\eps})\to \Sph$.

        Next, it is a straightforward check that the map~$\Psi_0$ commutes with the action~\eqref{eq:g-rho}. We will consider the map $\Psi_0$ as a map from the \emph{complex} neighbourhood $U_{2 \eps}(\Theta)$ of the central circle  $\Theta:=\{(\cznew,0):  |\cznew|=1\}$ in~$\C^2$. We will find an isotopy $\Psi_t$ between $\Psi_0$ and the identity in the space of (holomorphic) diffeomorphisms of $U_{2 \eps}(\Theta)$ onto their images, that commute with the action~\eqref{eq:g-rho}. Once such isotopy $\Psi_t$ is found, we will show that the compositions $\proj\circ \Psi_t$ provide an isotopy from $(\proj\circ \Psi_t)|_{M_\eps}$ to identity in the space of real $3$-dimensional diffeomorphisms onto their images.

        To find the isotopy, let us first post-compose $\Psi_0$ with the family of ``sliding'' maps
        \[
            S_t:(a,b)\mapsto (a,b-t a^\dg),\quad t\in [0,1],
        \]
        that are (global) diffeomorphisms of~$\C^2$, commuting with the action~\eqref{eq:g-rho}. This provides an isotopy $\Psi_t=S_t\circ \Psi_0$ between $\Psi_0$ and
        \[
            \Psi_1:(\cznew,\cwnew) \mapsto (\cznew(1+\cwnew \cznew^{-\dg}), \cznew^{\dg} \,(1-(1+\cwnew\cznew^{-\dg})^{\dg})),
        \]
        and for the second coordinate in the neighbourhood of $\Theta$ one has
        \[
            \cznew^{\dg} \, (1-(1+\cwnew\cznew^{-\dg})^{\dg}) =  -\dg \cwnew + O(z^2).
        \]
        Now, making another isotopy to rescale the second coordinate by a factor~$-\frac{1}{\dg}$,
        \begin{equation}\label{eq:psi2}
            \Psi_t= \left(\begin{matrix}
            1 & 0 \\
            0 & \frac{1}{\dg^{t-1}} \, e^{i\pi (t-1)}
            \end{matrix} \right)
             \circ \Psi_1, \quad t\in [1,2]
        \end{equation}
        we arrive at a map $\Psi_2$ that is close to the identity in a neighbourhood of~$\Theta$.

        Hence, the family
        \begin{equation}\label{eq:t-family}
            \Psi_t=(3-t)\Psi_2 + (t-2) \cdot \id, \quad t\in [2,3]
        \end{equation}
        is an isotopy between $\Psi_2$ and $\id$, consisting of diffeomorphisms on the image of $U_{\eps}(\Theta)$ when~$\eps$ is sufficiently small. As the action~\eqref{eq:g-rho} is linear,
        the map $\Psi_2$ and hence all the family~\eqref{eq:t-family} also commute with this action.

        Now, let $\Psi_t$, $t\in [0,3]$, be the isotopy of the holomorphic diffeomorphisms between $\Psi_0$ and $\Psi_3=\id$ that we have constructed. Consider the family $\proj \circ \Psi_t$ as maps on a sufficiently small neighbourhood~$M_{\eps}$. It suffices to check that this family is injective on the circle~$\Theta$ and that its 3-dimensional real differential $D (\proj \circ \Psi_t)$  stays non-degenerate on~$\Theta$ for all~$t$.

        We already know that the 4-dimensional real differential of $\Psi_t$ stays non-degenerate, and that $ \ker D(\proj)$ is one-dimensional (it is given by the tangent direction to the orbits of the action~\eqref{eq:g-rho} for real values of~$\rho$):
        \[
            \ker D(\proj)|_{\Psi_t(\cznew,0)} = T_{\Psi_t(\cznew,0)} \{g_\rho \circ \Psi_t (\cznew,0) \mid \rho \in \R_+\}.
        \]
        As $\Psi_t$ commutes with the action~\eqref{eq:g-rho}, this orbit is the $\Psi_t$-image of the orbit passing through the initial point~$(\cznew,0)$:
        \[
            \{g_\rho \circ \Psi_t (\cznew,0) \mid \rho \in \R_+\} =
                \Psi_t  \{ g_\rho(\cznew,0) \mid \rho \in \R_+\},
        \]
        and the latter is transverse to the $3$-sphere~$\Sph$, thus implying the non-degeneracy of the differential $D (\proj \circ \Psi_t)$ in restriction to $T_{(\cznew,0)} \Sph$. Finally, the injectivity check of $\Psi_t$ on~$\Theta$ is straightforward.

        \end{proof}

\begin{remark}
    Lemma~\ref{l:untwist} also follows from~\cite{MR2871329}, where tubular neighbourhoods of tori $T^p$ of higher dimension were studied; the map~\eqref{eq:hF-def} is the degree-$\dg$ ``favourite lifting'' of \cite[Example~4.2]{MR2871329}, and the existence of an isotopy to the identity follows from \cite[Proposition~4.4]{MR2871329}. Also, such maps were studied in~\cite{MR726853}. See also \cite[Chapter~4]{MR448362} for the general isotopy theory of tubular neighbourhoods. However, as we needed the description only in this particular case, we decided, for the reader's convenience, to give a straightforward argument.
\end{remark}

    As it was already mentioned, Lemma~\ref{l:untwist} together with the isotopy extension theorem imply that  each $\hF_j$ can be extended to a diffeomorphism $\tF_j:\Sph\to\Sph$. To conclude the construction of the example, we note that this choice of the extensions can be done in a particular way that ensures ``sufficient repulsion'' from~$M'$ to~$M$.
    \begin{lemma}\label{l:extension}
        There exists a choice of extensions $\tF_j$ of the maps $\hF_j:M\to M$ such that
        \begin{equation}\label{eq:covering}
            \Sph = \bigcup_{j=1}^s \tF_j^{-1}(\intrr M).
        \end{equation}
    \end{lemma}
    \begin{proof}
    Indeed,~\eqref{eq:covering} is equivalent to the statement that the intersection of the complements of the preimages is empty:
    \begin{equation}\label{eq:goal-nointersect}
        \bigcap_{j=1}^s \tF_j^{-1} (\Sph \setminus \intrr M)= \bigcap_{j=1}^s \tF_j^{-1} (M') = \emptyset.
    \end{equation}
    Consider an arbitrary choice of extensions $\tF_j$;
    as $\tF_j(M)\subset M$, we have
    \[
        K_j:=\tF_j^{-1}(M')\subset \intrr M'.
    \]
    If the sets $K_j$ do not intersect, there is nothing to do. Otherwise, the sets $K_1,K_2$ are compact subsets of the interior of the solid torus~$M'$. We claim that there exists a diffeomorphism~$\Phi$ satisfying
    \begin{equation}\label{eq:Phi-conditions}
        \Phi:M'\to M', \quad \Phi|_{\partial M'}= \id, \quad \Phi(K_1)\cap K_2 = \emptyset.
    \end{equation}
    It would then suffice to extend $\Phi$ to the whole sphere~$\Sph$ by the identity on $M$, and to replace~$\tF_1$ by $\tF_1\circ \Phi^{-1}$ (thus replacing $K_1$ by $\Phi(K_1)$ and hence ensuring the empty intersection~\eqref{eq:goal-nointersect}).

    To ensure that the diffeomorphism~$\Phi$ exists, take a vector field~$v$ on the disc~$\D$ that vanishes on the boundary and points towards the point $\cwnew_0=1\in \partial \D$:
    \[
        v(\cwnew) = \rho(\cwnew) \cdot (\cwnew_0-\cwnew),
    \]
    {where $\rho(\cwnew)=\exp(-\frac{1}{1-|\cwnew|^2})$.}

    \begin{figure}
        \centering
        \includegraphics[width=0.3\linewidth]{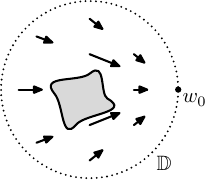}
        \caption{Vector field $v$ on the disc~$\D$: under its flow $\Phi_v^t$, the images of any compact set of $\intrr\D$ converges to the point $\cwnew_0\in \partial \D$.
        }
        \label{fig:v-on-D}
    \end{figure}
    Then, as $t\to \infty$, the images of any point $z\in \intrr \D$ under the time-$t$ map $\Phi_v^t$ converge to~$\cwnew_0$, and this convergence is uniform on compact sets (see Fig.~\ref{fig:v-on-D}). Now, let $K_1^{\D},K_2^{\D}\subset \D$ be the projections of $K_1$, $K_2$ on the first coordinate in the decomposition $M'=\D\times \Sc$. Then $K_1^{\D},K_2^{\D}\subset \intrr \D$, and hence for some sufficiently large $t$ one has $\Phi_v^t(K_1^{\D})\cap K_2^{\D}=\emptyset$: the image lies in an arbitrarily small neighbourhood of~$z_0$, and there is such a neighbourhood disjoint from~$K_2^{\D}$. We can then fix such a $t$ and take $\Phi=\Phi_v^t \times \id_{\Sc}$.
    \end{proof}

    Now, the conclusion~\eqref{eq:covering} of Lemma~\ref{l:extension} implies that for every point of $\Sph$ at least one of the maps sends this point inside~$M$ in a single step.
    Since $p_j>0$, with probability at least $\min_j p_j$ a random trajectory
    enters~$M$ at each step from any point of $\Sph\setminus M$; hence
    for every initial point $b_0\in\Sph$ the random trajectory
    $b_{n+1}=\tF_{\omega_n}(b_n)$ almost surely enters~$M$ after finitely
    many steps, and stays there since $\tF_j(M)\subset\mathrm{int}\,M$.

    Since all points almost surely enter $M$ in finitely many steps, all stationary measures are supported inside~$M$. The random dynamics, restricted to~$M$, is exactly that of Example~\ref{ex:torus}, thus implying all the claimed properties.
\end{proof}

\begin{proof}[Construction of Example~\ref{ex:rectangle}]
    Note that each $F_{i,j}$ is a smooth map with
    $F_{i,j}(\Pi)\subset\Pi$ and Lipschitz constant strictly less than~$1$, therefore the system is a uniformly contracting iterated function system. In particular, $\hmsp$ is the unique stationary measure of $(F_{i,j};p_{i,j})$.

    Now, if there were a proper closed invariant submanifold, it would contain the Cantor set $C$, the support of the unique stationary measure. Hence, it would be one-dimensional with a horizontal tangent line at the point $(0,0)\in C$. Thus, in a neighbourhood of $(0,0)$ it would be the graph of some function $x\mapsto r(x)$. However, no such graph can be invariant under both maps $F_{1,1}$ and $F_{1,2}$. Indeed, take any point $x_0\notin C$ in this neighbourhood; then, one has
    \[
        F_{1,1}(x_0,r(x_0)) = \left(\frac{x_0}{3}, \frac{r(x_0)}{2}-\theta(x_0)\right), \quad  F_{1,2}(x_0,r(x_0)) = \left(\frac{x_0}{3}, \frac{r(x_0)}{2}+\theta(x_0)\right).
    \]
    The graph would thus contain two points with the same $x$-coordinate $\frac{x_0}{3}$, whose $y$-coordinates would differ by $2\theta(x_0)\neq 0$. This provides the desired contradiction, showing that no proper closed invariant submanifolds exist.

\end{proof}

\begin{proof}[Construction of Example~\ref{ex:sphere2}]
    Embed $\Pi\subset\mathbb{S}^2$ and extend each $F_{i,j}$ to a
    diffeomorphism $\tF_{i,j}:\mathbb{S}^2\to\mathbb{S}^2$ as in the
    construction of Example~\ref{ex:sphere}, chosen so that
    $\mathbb{S}^2 = \bigcup_{i,j}\tF_{i,j}^{-1}(\mathrm{int}\,\Pi)$.
    Then every point almost surely enters~$\Pi$ in finitely many steps and stays there, reducing the claims to Example~\ref{ex:rectangle}.
\end{proof}

\section*{Acknowledgments}

The authors are grateful to Serge Cantat, \'Etienne Ghys, Anton Gorodetski, S\'ebastien Gou\"ezel, Ale Jan Homburg, Jeroen 
Lamb and Isabelle Liousse for the helpful discussions.

\bibliographystyle{plain}
\bibliography{references}
\end{document}